\documentclass[journal]{IEEEtran}
\usepackage{amsmath}
\usepackage{amsthm}
\usepackage{wrapfig}
\usepackage{amssymb}
\usepackage{lineno,hyperref}
\usepackage{cases}
\usepackage{epsfig}
\usepackage{graphicx}
\usepackage{subfig}
\usepackage{mathrsfs}
\usepackage{colortbl}
\usepackage{graphicx}
\usepackage{epstopdf}
\usepackage{cite}
\usepackage{bm}
\modulolinenumbers[5]

\newtheorem{lem}{Lemma}
\newtheorem{thm}{Theorem}

\newtheorem{ass}{Assumption}

\newtheorem{rem}{Remark}

\begin{document}

\title{LQG Differential Stackelberg Game under Nested Observation Information
Pattern}

\author{Zhipeng Li, Dami\'{a}n Marelli$^{*}$,  Minyue Fu,\emph{~Fellow,~IEEE}, and  Huanshui Zhang,\emph{~Senior Member,~IEEE}
\thanks{The work was supported by the National Natural
 Science Foundation of China (Grant Nos. 61633014, U1701264 and 61803101), Australian Research Council under Grant DP200103507, and the Argentinean Agency for Scientific and Technological Promotion (PICT-201-0985).}
\thanks{Zhipeng Li is with College of Engineering, Qufu Normal University, Rizhao, 276826, China.}
\thanks{Damian Marelli is with School of Guangdong University of Technology, Guangzhou, 510006, China. He is also with French Argentine
International Center for Information and Systems Sciences, National Scientific and Technical Research Council, Argentina (damianenviena@gmail.com).}
\thanks{Minyue Fu is with School of Electrical Engineering and Computing, University of Newcastle, NSW 2308, Australia. }
\thanks{Huanshui Zhang is with School of Control Science and Engineering, Shandong
University, Jinan, 250061, China.}
\thanks{$^{*}$Corresponding author}
}

\maketitle
\begin{abstract}
  We investigate the linear quadratic Gaussian Stackelberg game under a class of nested observation information pattern. Two decision makers implement control strategies relying on different information sets: The follower uses its observation data to design its strategy, whereas the leader implements its strategy using global observation data. We show that the solution requires solving a new type of forward-backward stochastic differential equations whose drift terms contain two types of conditional expectation terms associated to the adjoint variables. We then propose a method to find the functional relations between each adjoint pair, i.e., each pair formed by an adjoint variable and the conditional expectation of its associated state. The proposed method follows a layered pattern. More precisely, in the inner layer, we seek the functional relation for the adjoint pair under the $\sigma$-sub-algebra generated by follower's observation information; and in the outer layer, we look for the functional relation for the adjoint pair under the $\sigma$-sub-algebra generated by leader's observation information. Our result shows that the optimal open-loop solution admits an explicit feedback type representation. More precisely, the feedback coefficient matrices satisfy tuples of coupled forward-backward differential Riccati equations, and feedback variables are computed by Kalman-Bucy filtering.
\end{abstract}


\section{Introduction}

Stackelberg game was introduced by H. Von Stackelberg as a static game in the context of static economic equilibrium in 1934 \cite{Stackelberg1952}, when he defined a concept of a hierarchical solution for markets where some firms have power of domination over others. This solution is now known as the Stackelberg equilibrium which, in terms of two-person nonzero-sum games, involves two players with asymmetric roles, one leader and one follower. On the other hand, the study of differential games was initiated by Isaacs in 1954 \cite{Isaacs1954}. Different from the static optimization problem, the related equality constraint becomes a differential dynamic system.  With the rise of differential games, the differential Stackelberg game entered the control literature through the early works ~\cite{Cruz1978,Medanic1978Closed,Basar1980,Dockner2000,Yong2002}. Since then, stochastic differential Stackelberg games have been investigated by many authors and have been used in many applications. For example, Castanon {\em et. al.}~\cite{Castanon1975} gave explicit solutions to the linear quadratic Gaussian (LQG) Stackelberg game under asymmetric information where the follower knows the leader's control; Oksendal {\em et. al.}~\cite{Oksendal2014} derived a maximum principle for the controlled jump diffusion process and applied it to the news vendor problem; Bensoussan {\em et. al.}~\cite{Bensoussan2015} derived a global maximum principle for both open-loop and closed-loop Stackelberg differential games. Xu {\em et. al.}~\cite{2016Sufficient} researched the Stackelberg strategy for the time-delay Stackelberg game in discrete-time dynamic setting and presented a necessary and sufficient condition. Mukaidani {\em et. al.}~\cite{2017Infinite} investigated the infinite horizon linear quadratic discrete-time Stackelberg game problem with multiple decision makers and derived the necessary conditions for the existence of the optimal strategy set.  Moon {\em et. al.} ~\cite{2015Linear} considered a class of stochastic differential games with the Stackelberg mode of play, with one leader and $N$ uniform followers and designed an approximate Stackelberg equilibrium for the games with sufficiently many finite agents. Lin~{\em et. al.}~\cite{2019An}~concerned an open-loop linear quadratic Stackelberg game of the mean-field stochastic systems, and shown that the open-loop Stackelberg equilibrium admits a feedback representation involving a new state and its mean.

Recently, Shi {\em et. al.}~\cite{ShiLeader} investigated the stochastic differential Stackelberg game for nonlinear systems under asymmetric information, i.e., the leader and the follower implement control strategies using different information patterns. By Applying the Girsanov's theorem, the maximum principle is derived for this problem and solutions for control strategies are given.  In this paper, we aim to develop more explicit solutions for LQG Stackelberg game under nested observation information pattern. This pattern implies that the information available to the leader contains the information available to the follower. In our model, the leader is able to access two linear noisy observation data and the follower is only able to access one of them.

The difference between our work and the one in Shi {\em et. al.}~\cite{ShiLeader} is mainly on the following aspect. The control strategy processes of their work are adapted to the filtration generated by two Wiener processes, i.e., they consider the case of the non-anticipative information pattern. In contrast, we deal with the noisy information pattern~\cite{Charalambos2017Noisy,Charalambos2017Centralized,Charalambos2017application}, i.e., two decision makers implement strategies relying on two observation processes, respectively. This leads to our estimators being different from ones in~\cite{ShiLeader}. In theirs situation, two decision makers' estimators are obtained directly by taking conditional expectation with respect to the $\sigma$-sub-algebras generated by two Wiener processes stated above respectively. In ours case, it appears an interesting problem, i.e., if the related estimators are Kalman-Bucy type or not. In Castanon {\em et. al.}~\cite{Castanon1975}, the follower uses its observation and the leader's strategy to design its own strategy, whereas the leader implements its strategy using global observation data, thus the information set of the follower to implement strategy contains the observation itself and a subset of both decision makers' observation data. Comparing with \cite{Castanon1975}, the follower in our model uses only the observation data to implement its strategy. In this setting, the separation principle is no longer valid~\cite{Wonham1968}. Different from the method in~\cite{ShiLeader}, we take the dynamic system equation and the observation equation as equality constraints to overcome this difficulty.  A comparison between this work and the previous one \cite{Li2020} is as follows:
In both works the follower implements its strategy relying on its observation data, and this data is also available to the leader for implementing its strategy. In \cite{Li2020} the leader also knows the system state information. In this work we generalize this setup by making available to the leader only its local observation data. That is, the leader is no longer able to access the system state information. In this sense, the setup studied in this work is a generalization of the one studied in \cite{Li2020}.
In addition to the above point, the technical results derived in the current paper are more complete than those in \cite{Li2020}. In this work we provide conditions to guarantee existence of the solution, and give sufficient and necessary condition for the optimal open-loop solution. In contrast with the current paper, we only discuss the necessary condition for the optimal open-loop solution. Neither the existence condition and the sufficient condition for the optimal open-loop solution nor the existence and uniqueness analysis for the solution of the related FB-SDE are given in the previous paper. Finally, technically speaking, the different setup studied in this work leads to a FB-SDE which is different from the one in \cite{Li2020}. More precisely, it contains two kind of conditional expectation terms in its drift, rather than only one kind. As a result, the layered analysis is technically more complex, because two kind of (non-independent) innovation processes appear.

With respect to this class of  FB-SDEs and their solving methods, as we know, there are  few reports. We give a layered calculation method to solve it. Noticing that there are two kind of conditional mean terms of of adjoint states, we then seek the functional relation between these two terms and their corresponding state variables respectively. This procedure is done by a layered way. We first look for the functional relation for the adjoint pair under the $\sigma$-sub-algebra generated by follower's observation information in the inner layer part; then we look for the functional relation for the adjoint pair under the $\sigma$-sub-algebra generated by leader's observation information in the outer layer part.

The main innovations are the following.
\begin{itemize}
  \item Show that the solution of the LQG Stackelberg game problem under nested observation information pattern studied requires solving a new type of FB-SDEs.
  \item Propose a new layered calculation method to find the feedback type strategies for both decision makers.
\end{itemize}

The rest of the paper is organized as follows. In Section~\uppercase\expandafter{\romannumeral2}, the LQG Stackelberg game under nested information patterns is formulated. In Section~\uppercase\expandafter{\romannumeral3}, the sufficient and necessary conditions for the follower and the leader to be optimal are given respectively. In addition, the existence and uniqueness property for the solution of the FB-SDE appeared in leader's problem is derived. In section~\uppercase\expandafter{\romannumeral4}, we deal with the computation aspect and obtain the feedback type strategies. In section~\uppercase\expandafter{\romannumeral5}, we give some examples to show its application of our main result.

\textbf{Notation:}
\emph{We use $M^{'}$ to denote the transpose of matrix or vector $M$ and use  $I_{id}$ to denote the identity matrix. $\mathrm{diag}(A_{11},..,A_{nn})$ denotes the block diagonal matrix whose diagonal entries are $A_{11},..,A_{nn}$, and $ \mathrm{col}(A_{11};..;A_{n1}):=(A_{11}^{'},..,A_{n1}^{'})^{'}$. We also use  $S\succ 0$ and $S\succeq 0$ to denote the symmetric matrix $S$ to be positive definite  and positive semi-definite respectively.
$\langle \cdot,\cdot\rangle$ and $|\cdot|$ denote
the inner product and the norm in the Euclidean space, respectively.  For a given complete probability space $(\Omega, \mathcal{F}, \mathbb{P})$,
$\mathbb{E}[\cdot]$ and $\mathrm{cov}(\cdot,\cdot)$ denote the expectation and covariance operators defined by the probability measure $\mathbb{P}$ respectively.}

\section{Problem Formulation}

\subsection{Model Description}
We study the linear quadratic Stackelberg game described by
\begin{equation}\label{eq2.1a}
\begin{aligned}
\mathcal{J}^{i}(U^{L},U^{F})&=\mathbb{E}\int_{0}^{T}\frac{1}{2}\Big(\langle Q_{i}X_{t}, X_{t}\rangle
+\langle R_{iL}U_{t}^{L}, U_{t}^{L}\rangle\\
&+\langle R_{iF}U_{t}^{F}, U_{t}^{F}\rangle\Big) dt
+\mathbb{E}\frac{1}{2}\langle G_{i} X_{T}, X_{T}\rangle,
\end{aligned}
\end{equation}
subject to the signal model
\begin{equation}\label{eq2.1b}
\begin{aligned}
\begin{cases}
dX_{t}=(AX_{t}+B_{F}U_{t}^{F}+B_{L}U_{t}^{L})dt+D dW_{t},\\
dZ_{t}^{1}=H_{1}X_{t}dt+dV_{t}^{1},\\
dZ_{t}^{2}=H_{2}X_{t}dt+dV_{t}^{2},
\end{cases}
\end{aligned}
\end{equation}
where $i=F,L$. $X_t\in \mathbb{R}^{n}$ is the state of the linear time-invariant system, $Z_{t}^{1}\in \mathbb{R}^{l}$ and $Z_{t}^{2}\in\mathbb{R}^{m}$ are the linear noisy observations at time $t$, initialized by $Z_{0}^{1}=0$ and $Z_{0}^{2}=0$ respectively. Between them, the follower is able to access to the observation process $\{Z^{2}_{t}\}$, whereas the leader is able to access to both $\{Z^{1}_{t}\}$ and $\{Z^{2}_{t}\}$. $U_t^{F}\in \mathbb{U}^{F}\subset\mathbb{R}^{k}$ and $U_t^{L}\in \mathbb{U}^{L}\subset\mathbb{R}^{d}$ are the strategies of two decision makers, respectively, where $\mathbb{U}^{F}$ and $\mathbb{U}^{L}$ are convex and compact. $(W^{'},V^{1'},V^{2'})^{'}\in \mathbb{R}^{w}$ is the standard $\{\mathcal{F}_{t}\}_{t\in[0,T]}$-adapted Wiener process on the filtered complete probability space $(\Omega, \mathcal{F}, \{\mathcal{F}_{t}\}_{t\in[0,T]}, \mathbb{P})$. The initial value $X_{0}=\xi$ is a Gaussian random variable which is independent of the standard Wiener process above. $A, D, H_{1}, H_{2}$, $B_{i}$ are constant matrices with adequate dimensions respectively; Weights matrices $Q_{i}$, $R_{ij}$, $G_{i}$, $Q_{i}$, $R_{ii}$ are symmetric matrices, with $i,j=F,L$.  The detail conditions on them are listed in Assumption~\ref{ass}.

The two information patterns available to the decision makers $L$ and $F$ are denoted by two filtrations $\{\mathcal{I}_{t}^{L}\}_{t\in[0,T]}$ and $\{\mathcal{I}_{t}^{F}\}_{t\in[0,T]}$ respectively. In the above, for arbitrary $t\in [0,T]$, $\mathcal{I}_{t}^{L}$ and $\mathcal{I}_{t}^{F}$ are $\sigma$-sub-algebras of $\mathcal{F}_{t}$.

The admissible strategy sets of the follower and the leader are defined, respectively, by
\begin{align*}
\mathcal{U}^{F}&:=\Big\{U^{F}~|~U^{F}_{t}~\mathrm{is~} \mathbb{U}^{F}-\mathrm{valued}~\mathrm{and}~
\{\mathcal{I}_{t}^{F}\}-\mathrm{adapted~}\\
  &~~~~\mathrm{process}~\mathrm{~with~}
\mathbb{E}\int_{0}^{T}\left|U^{F}_{t}\right|^{2}dt<\infty\Big\},\\
\mathcal{U}^{L}&:=\Big\{U^{L}~|~U^{L}_{t}~\mathrm{is~}
\mathbb{U}^{L}-\mathrm{valued}~\mathrm{and}~
\{\mathcal{I}_{t}^{L}\}-\mathrm{adapted~}\\
  &~~~~\mathrm{process}~\mathrm{~with~}
\mathbb{E}\int_{0}^{T}\left|U^{L}_{t}\right|^{2}dt<\infty\Big\}.
\end{align*}

The open-loop LQG Stackelberg game under asymmetric information pattern is stated as follows:
\emph{Determine the control functions $U^{F\star}\in \mathcal{U}^{F}$ and $U^{L\star}\in \mathcal{U}^{L}$ such that $U^{F\star}=\eta\left( U^{L\star}\right)$, with $\eta$ being the reaction function in the sense that the follower responses to the leader's strategy.
$U^{L\star}$ and the mapping
 $\eta : \mathcal{U}^{L}\rightarrow \mathcal{U}^{F}$ satisfy the following conditions:}
\begin{equation}\label{eq2.2}
\begin{aligned}
&(Follower's ~Optimality~ Condition)\\
&\mathcal{J}^{F}\left(U^{L},\eta \left( U^{L}\right)\right)\leq \mathcal{J}^{F}\left(U^{L},U^{F}\right),~~\forall~U^{L}\in \mathcal{U}^{L},\\
&~~~~~~~~~~~~~~~~~~~~~~~~~~~~~~~~~~~~~~~~~~U^{F}\in \mathcal{U}^{F},\\
&(Leader's~ Optimality ~Condition)\\
&\mathcal{J}^{L}(U^{L\star},\eta \left( U^{L\star}\right))\leq \mathcal{J}^{L}(U^{L}, \eta \left( U^{L}\right)),~~\forall~U^{L}\in\mathcal{U}^{L}.
\end{aligned}
\end{equation}

\begin{ass}\label{ass}
Suppose the weight matrices satisfy the following conditions:
\begin{description}
  \item[\textbf{A1a}.] $Q_{F}\succeq 0,~~G_{F}\succeq0,~~R_{FF}\succ0$,\par
  \item[\textbf{A1b}.] $Q_{L}\succeq 0,~~G_{L}\succeq0,~~R_{LL}\succ0,~~R_{LF}\succ 0$.
\end{description}

\end{ass}

With regard to (\ref{eq2.1b}), we consider an additive noise linear system rather than a one with multiplicative noise. As is well known, the filter of the latter case is nonlinear.

\subsection{Optimization Problem}

As stated above, the follower is able to access to the observation process $\{Z^{2}_{t}\}$, whereas the leader is able to access to both $\{Z^{1}_{t}\}$ and $\{Z^{2}_{t}\}$. We suppose the follower and the leader depend respectively on their respective information available to implement their strategies at time $t$, $t\in[0,T]$. To describe it, let $\mathcal{I}_{t}^{L}=\mathcal{Z}_{t}^{L}$, $\mathcal{I}_{t}^{F}=\mathcal{Z}_{t}^{F}$.  They are generated by their respective information sets, for detail
\begin{align*}
\mathcal{Z}_{t}^{L}:&=\sigma (Z_{s}^{1},Z_{s}^{2}~|~0\leq s\leq t),\\
\mathcal{Z}_{t}^{F}:&=\sigma (Z_{s}^{2}~|~0\leq s\leq t), ~~t\in[0,T].
\end{align*}
Notice that, for each $t\in[0,T]$, $\mathcal{Z}_{t}^{F}\subset \mathcal{Z}_{t}^{L}$, implying that the $\sigma$-algebras are nested \cite{Castanon1975}. For other works in linear quadratic control under nested information pattern  we refer the reader to \cite{2020Nayyar}.

\textbf{The Follower's Problem}~
Based on the \textit{Follower's Optimality Condition}
(\ref{eq2.2}), given the leader's strategy $U^{L}\in \mathcal{U}^{L}$, the follower is faced with the following optimal control problem:
\begin{equation}\label{eq3.1}
\begin{aligned}
\inf_{U^{F}\in \mathcal{U}^{F}}\mathcal{J}^{F}(U^{L},U^{F}),
\end{aligned}
\end{equation}
subject to
\begin{equation}\label{eq3.2}
\begin{aligned}
\begin{cases}
dX_{t}=(AX_{t}+B_{F}U_{t}^{F}+B_{L}U_{t}^{L})dt+D dW_{t},\\
dZ_{t}^{2}=H_{2}X_{t}dt+dV_{t}^{2},\\
~~X_{0}=\xi,~~Z_{0}^{2}=0.
\end{cases}
\end{aligned}
\end{equation}

\textbf{The Leader's Problem}
~Based on the \textit{Leader's Optimality Condition} of
(\ref{eq2.2}), after substituting $U^{F\star}$ obtained from solving the follower's problem into (\ref{eq3.2}), we get the optimization problem below faced by the leader:
\begin{equation}\label{eq3.5}
\begin{aligned}
\inf_{U^{L}\in \mathcal{U}^{L}}\mathcal{J}^{L}(U^{L},U^{F\star}),
\end{aligned}
\end{equation}
subject to
\begin{equation}\label{eq3.6}
\begin{aligned}
\begin{cases}
dX_{t}=\left(AX_{t}+B_{F}U^{F\star}+B_{L}U_{t}^{L}\right)dt+D dW_{t},\\
dZ_{t}^{1}=H_{1}X_{t}dt+dV_{t}^{1},\\
dZ_{t}^{2}=H_{2}X_{t}dt+dV_{t}^{2},\\
~~X_{0}=\xi,~~Z_{0}^{1}=0,~~Z_{0}^{2}=0.
\end{cases}
\end{aligned}
\end{equation}

\section{Optimality Conditions}

For the follower's problem, for every $U^{L}\in \mathcal{U}^{L}$, it is a classic LQG regulation problem under assumption \textbf{A1a}. Thus, the optimization problem of the follower's is well-posed and  has a unique solution~\cite{Ahmed1998Linear,Bensoussan1992Stochastic}.
\begin{lem}\label{lem1}
Let assumption \textbf{A1a} be satisfied. Then, a
necessary and sufficient condition for $U^{F}$ to be an open-loop optimal strategy of the follower's problem described by (\ref{eq3.1})-(\ref{eq3.2}) is
\begin{equation}\label{eq3.3}
\begin{aligned}
U^{F\star}_{t}=-R_{FF}^{-1}B^{'}_{F}\hat{p}_{t}^{F},~~~~t\in[0,T],
\end{aligned}
\end{equation}
with $\hat{p}_{t}^{F}:=\mathbb{E}\left[p_{t}^{F}|\mathcal{Z}_{t}^{F}\right]$. Here, $\left(p_{t}^{F},\left(q^{F}_{t}, r^{F}_{t}\right)\right)$ (called the adjoint process) is the unique solution of the BSDE (conventionally called the adjoint equation):
\begin{equation}\label{eq3.4}
\begin{aligned}
\begin{cases}
dp^{F}_{t}=-\Big[A^{'}p^{F}_{t}+Q_{F}X_{t}\Big]dt+q^{F}_{t}dW_{t}
+r^{F}_{t}dV_{t}^{2},\\
~~p^{F}_{T}=G_{F}X_{T}.
\end{cases}
\end{aligned}
\end{equation}
\end{lem}
Notice that there appears a adjoint process when solve the follower's problem, a similar situation may appear later in leader's problem, see (\ref{eq3.9}). Its proof is similar to the work in \cite{Charalambos2017Noisy}.  For easy of reading, we give its proof in Appendix.

\bigskip

After substituting $U^{F\star}$ in (\ref{eq3.3}) into (\ref{eq3.2}), the optimization problem  faced by the leader changes to:
\begin{equation}\label{eq3.5}
\begin{aligned}
\inf_{U^{L}\in \mathcal{U}^{L}}\mathcal{J}^{L}(U^{L},U^{F\star}),
\end{aligned}
\end{equation}
subject to
\begin{equation}\label{eq3.6}
\begin{aligned}
\begin{cases}
dX_{t}=\left(AX_{t}-B_{F}R_{F}^{-1}B_{F}^{'}\hat{p}_{t}^{F}+B_{L}U_{t}^{L}\right)dt\\
~~~~~~~~~~~+D dW_{t},\\
dZ_{t}^{1}=H_{1}X_{t}dt+dV_{t}^{1},\\
dZ_{t}^{2}=H_{2}X_{t}dt+dV_{t}^{2},\\
dp^{F}_{t}=-\left[A^{'}p^{F}_{t}+Q_{F}X_{t}\right]dt+q^{F}_{t}dW_{t}\\
~~~~~~~~~~~+r^{F}_{t}dV_{t}^{2}.\\
~~X_{0}=\xi,~~Z_{0}^{1}=0,~~Z_{0}^{2}=0;~~
p^{F}_{T}=G_{F}X_{T}.
\end{cases}
\end{aligned}
\end{equation}

Notice that the existence and the uniqueness of the follower's strategy implies that, for every $U^{L}\in \mathcal{U}^{L}$, there exists a solution for (\ref{eq3.6}). Suppose $(X^1, (p^{F1}, q^{F1}, r^{F1}))$ and $(X^2, (p^{F2}, q^{F2}, r^{F2}))$ are two different solutions of (\ref{eq3.6}). Let
\begin{equation}\label{error}
\begin{aligned}
        &\phi_{t}(X):=X_{t}^{1}-X_{t}^{2}, ~~\phi_{t}(p^{F}):=p_{t}^{F1}-p_{t}^{F2},\\
&\phi_{t}(q^{F}):=q^{F1}_{t}-q^{F2}_{t}, ~~\phi_{t}(r^{F}):=r^{F1}_{t}-r^{F2}_{t}.
      \end{aligned}
\end{equation}
Suppose assumption \textbf{A1a} holds. From an argument similar to the one in \cite{Tang2003}, it follows that
\begin{equation}\label{ineq2}
\begin{aligned}
&\mathbb{E}\sup_{t\in[0,T]}\left|\phi_{t}(X)\right|^{2}+\mathbb{E}\sup_{t\in[0,T]}
\left|\phi_{t}(p^{F})\right|^{2}\\
&~~+\mathbb{E}\int_{0}^{T}\left[\left|\phi_{t}(q^{F})\right|^{2}
+\left|\phi_{t}(r^{F})\right|^{2}\right]dt\\
&\leq K\left(\mathbb{E}\left|\phi_{0}(X)\right|^{2}+\mathbb{E}\int_{0}^{T}
\left|\phi_{t}(U^{L})\right|^{2}dt\right),
\end{aligned}
\end{equation}
where  $K$ is a positive constant and $\phi_{t}(U^{L}):=U_{t}^{L1}-U_{t}^{L2}$ denotes the error of two different strategies of the leader.

Notice that (\ref{ineq2}) implies all left terms can be dominated by initial data error and the energy of the strategy process error. Based on this, uniqueness of solution of (\ref{eq3.6}) is derived.
From now on, we have got, for every $U^{L}\in \mathcal{U}^{L}$, that the FB-SDE in (\ref{eq3.6}) has a unique solution. To proceed, we need a lemma.

\begin{lem}
Let assumptions \textbf{A1a}, \textbf{A1b} be satisfied. Then the leader's problem has an optimal solution.
\end{lem}
\begin{proof}
Let $\left(X_{\cdot,n}, p^{F}_{\cdot,n}, q^{F}_{\cdot,n}, r^{F}_{\cdot,n}, U^{L}_{\cdot,n}\right)$ be a minimizing sequence, i.e., $\mathcal{J}^{L}(U^{L}_{n},U^{F\star})\rightarrow\inf_{U^{L}\in \mathcal{U}^{L}}\mathcal{J}^{L}(U^{L},U^{F\star})$ as $n\rightarrow +\infty$, where $\left(X_{\cdot,n}, p^{F}_{\cdot,n}, q^{F}_{\cdot,n}, r^{F}_{\cdot,n}\right)$ denotes the state processes of (\ref{eq3.6}) corresponding with the strategy process $U^{L}_{\cdot,n}$ for every $n\in \mathbb{N}$. From the boundness of $\mathbb{U}^{L}$, we have
\begin{align*}
\mathbb{E}\int_{0}^{T}\left|U^{L}_{t,n}\right|^{2}dt\leq K,~~~~\forall~n\in \mathbb{N},
\end{align*}
with $K$ being a positive constant.  Thus, there is a subsequence $\{U^{L}_{\cdot,k}\}$ such that
$U^{L}_{\cdot,k}$ converges weakly to $\overline{U}^{L}_{\cdot}$ as $k\rightarrow +\infty$ in the square square-integrable space containing all $\{\mathcal{Z}_{t}^{L}\}$-adapted $\mathbb{R}^{d}$-valued random processes.

By Mazur theorem~\cite{Haim2011}, there exists a
sequence $\widetilde{U}^{L}_{\cdot,n}$ made up of convex combinations of the $U^{L}_{\cdot,k}$'s that converges strongly to $\overline{U}^{L}_{\cdot}$:
\begin{align*}
\widetilde{U}^{L}_{\cdot,n}=\sum_{k\geq 1}a_{nk}U^{L}_{\cdot,k},~~a_{nk}\geq 0,~~\sum_{k\geq 1}a_{nk}=1.
\end{align*}
Since the subset $\mathbb{U}^{L}\subset \mathbb{R}^{d}$ is convex and closed, it follows that
$\overline{U}^{L} \in \mathcal{U}^{L}$. Let $\left(\widetilde{X}_{\cdot,n}, \widetilde{p}^{F}_{\cdot,n}, \widetilde{q}^{F}_{\cdot,n}, \widetilde{r}^{F}_{\cdot,n}\right)$ and $\left(\overline{X}_{\cdot}, \overline{p}^{F}_{\cdot}, \overline{q}^{F}_{\cdot}, \overline{r}^{F}_{\cdot}\right)$ be the states under the strategies $\widetilde{U}^{L}_{\cdot,n}$ and $\overline{U}^{L}_{\cdot}$ respectively, then it follows, from (\ref{ineq2}),
that
\begin{align*}
&\mathbb{E}\sup_{t\in[0,T]}\left|\widetilde{X}_{t,n}-\overline{X}_{t}\right|^{2}=0,\\ &\mathbb{E}\sup_{t\in[0,T]}\left|\widetilde{p}^{F}_{t,n}-\overline{p}^{F}_{t}\right|^{2}=0,~~n\rightarrow +\infty.
\end{align*}

Finally, from the convexity of the quadratic performance index of (\ref{eq3.5}) with their respective variables, we have
\begin{align*}
\mathcal{J}^{L}(\overline{U}^{L},&-R_{FF}^{-1}B_{F}^{'}\hat{\overline{p}}_{t}^{F})
=\lim_{n\rightarrow+\infty}\mathcal{J}^{L}(\widetilde{U}^{L}_{n},
-R_{FF}^{-1}B_{F}^{'}\hat{\widetilde{p}}_{t,n}^{F})\\
&\leq \lim_{n\rightarrow+\infty}\sum_{k\geq 1}a_{nk}\mathcal{J}^{L}(U^{L}_{k},
-R_{FF}^{-1}B_{F}^{'}\hat{p}_{t,k}^{F})\\
&=\inf_{U^{L}\in \mathcal{U}^{L}}\mathcal{J}^{L}(U^{L},U^{F\star}).
\end{align*}
Hence, $\overline{U}^{L}$ is optimal.

\end{proof}

\begin{rem}
This lemma develops the result of stochastic LQ problems in \cite{Yong1999}, pp.68. 
With regard to the uniqueness of the optimal control, we show later the optimal control is computed by a FB-SDE. Under our assumption, we show the FB-SDE admits a unique solution. Therefore, the optimal control is unique.
\end{rem}

Next, we introduce a new FB-SDE (\ref{eq3.9}).  This FB-SDE is needed for the proof of Lemma~\ref{lem2}.  We write it here in advance and discuss the existence and  uniqueness property of its solution to avoid the proof of Lemma~\ref{lem2} being too long. Comparing with (\ref{eq3.6}), there are two additional equations appear.
They are the equations satisfied by two adjoint state processes, $(p^{L}, k^{L}, r^{L})$ and $Y$, of the original state processes $X$ and $(p^{F},q^{F},r^{F})$ respectively.

\begin{equation}\label{eq3.9}
\begin{aligned}
\begin{cases}
dp^{L}_{t}=-\left[A^{'}p^{L}_{t}+Q_{F}Y_{t}+Q_{L}X_{t}\right]dt+q^{L}_{t}dW_{t}\\
~~~~~~~~~~~+k^{L}_{t}dV_{t}^{1}+r_{t}^{L}dV_{t}^{2},\\
dp^{F}_{t}=-\left[A^{'}p^{F}_{t}+Q_{F}X_{t}\right]dt+q^{F}_{t}dW_{t}\\
~~~~~~~~~~~+r^{F}_{t}dV_{t}^{2},\\
~dX_{t}=\left(AX_{t}+B_{L}U_{t}^{L\star}-B_{F}R_{FF}^{-1}B_{F}^{'}\hat{p}_{t}^{F}\right)dt\\
~~~~~~~~~~~+D dW_{t},\\
U_{t}^{L\star}=-R_{L}^{-1}B_{L}^{'}\tilde{p}^{L}_{t},~~~~t\in [0,T],\\
~dY_{t}=\Big(AY_{t}-B_{F}R_{FF}^{-1}B_{F}^{'}\hat{p}^{L}_{t}\\
~~~~~~~~~~~~+B_{F}R_{FF}^{-1}R_{LF}
R_{FF}^{-1}B_{F}^{'}\hat{p}_{t}^{F}\Big)dt,\\
~dZ_{t}^{1}=H_{1}X_{t}dt+dV_{t}^{1},\\
~dZ_{t}^{2}=H_{2}X_{t}dt+dV_{t}^{2},\\
~~~p^{L}_{T}=G_{L}X_{T}+G_{F}Y_{T},~~p^{F}_{T}=G_{F}X_{T}, \\
~~~X_{0}=\xi,~~Y_{0}=0,~~Z_{0}^{1}=0,~~Z_{0}^{2}=0,
\end{cases}
\end{aligned}
\end{equation}
where $\tilde{p}_{t}^{L}:=\mathbb{E}\left[p_{t}^{L}|\mathcal{Z}_{t}^{L}\right]$,  $\hat{p}_{t}^{F}:=\mathbb{E}\left[p_{t}^{F}|\mathcal{Z}_{t}^{F}\right]$ and $\hat{p}_{t}^{L}:=\mathbb{E}\left[p_{t}^{L}|\mathcal{Z}_{t}^{F}\right]$.

\begin{lem}\label{lem0001}
Let assumptions \textbf{A1a}, \textbf{A1b} be satisfied. Then, the FB-SDE (\ref{eq3.9}) has a unique solution.
\end{lem}
\begin{proof}
If  the solution of (\ref{eq3.9}) does not exist, then the optimal strategy $U^{L\star}$ does not exist. Assumptions \textbf{A1a} and  \textbf{A1b} imply the existence of
an optimal strategy $U^{L\star}$. It will be shown in the following Lemma~\ref{lem2} that $U^{L\star}$ should satisfy (\ref{eq3.7}) and then $(X, Y, p^{L},q^{L}, k^{L},r^{L}, p^{F}, q^{F}, r^{F} )$ is
a solution. Let's assume this proceed has been done temporarily, then the existence part is proved. To show the
uniqueness part, we need a series of priori estimates. In the following, $\phi_{t}(X)$, $\phi_{t}(Y)$, $\phi_{t}(p^{L})$, $\phi_{t}(q^{L})$, $\phi_{t}(k^{L})$, $\phi_{t}(r^{L})$,
$\phi_{t}(p^{F})$, $\phi_{t}(q^{F})$,  $\phi_{t}(r^{F})$ are defined similarly as (\ref{error}), all of them denote the deviation of two tuple of different solutions of (\ref{eq3.9}).  We first show the uniqueness of $X$, which relies on  the following two priori estimates:
\begin{equation}\label{dual2}
\begin{aligned}
&\mathbb{E}\langle G_{L}\phi_{T}(X), \phi_{T}(X)\rangle\\
&+\mathbb{E}\int_{0}^{T}\Big(\langle B_{L}R_{LL}^{-1}B_{L}^{'}\phi_{t}(\tilde{p}^{L}), \phi_{t}(\tilde{p}^{L})\rangle\\
&+\langle Q_{L}\phi_{t}(X),\phi_{t}(X)\rangle\\
&+\langle B_{F}R_{FF}^{-1}R_{LF}
R_{FF}^{-1}B_{F}^{'}\phi_{t}(\hat{p}^{F}),\phi_{t}(\hat{p}^{F})\rangle\Big)dt\\
&=\mathbb{E}\langle \phi_{0}(X), \phi_{0}(p^{L})\rangle\leq \mathbb{E}(|\phi_{0}(X)|\cdot|\phi_{0}(p^{L})|),
\end{aligned}
\end{equation}
and
\begin{equation}\label{fores2}
\begin{aligned}
&\mathbb{E}\int_{0}^{T}e^{-Kt}\left|\phi_{t}(X)\right|^{2}dt\leq \mathbb{E}\left|\phi_{0}(X)\right|^{2}\\
&+\mathbb{E}\int_{0}^{T}\langle e^{-Kt}\phi_{t}(\tilde{p}^{L}), B_{L}R_{LL}^{-1}B_{L}^{'}\phi_{t}(\tilde{p}^{L})\rangle dt\\
&+\mathbb{E}\int_{0}^{T}\langle e^{-Kt}\phi_{t}(\hat{p}^{F}),B_{F}R_{FF}^{-1}B_{F}^{'}\phi_{t}(\hat{p}^{F})\rangle dt,
\end{aligned}
\end{equation}
where they are obtained by applying It\^{o}'s formula to
$\langle \phi_{t}(X), \phi_{t}(p^{L})\rangle-\langle \phi_{t}(Y), \phi_{t}(p^{F})\rangle$ and $e^{-Kt}|\phi_{t}(X)|^{2}$ and then taking expectation respectively. We provide the latter one as follow:
\begin{equation}\label{fores1}
\begin{aligned}
&\mathbb{E}e^{-KT}\left|\phi_{T}(X)\right|^{2}+\mathbb{E}\int_{0}^{T}Ke^{-Kt}\left|\phi_{t}(X)\right|^{2}dt\\
&=\mathbb{E}\left|\phi_{0}(X)\right|^{2}+\mathbb{E}\int_{0}^{T}\Big\langle 2e^{-Kt}\phi_{t}(X), A\phi_{t}(X)\\
&~~-B_{L}R_{LL}^{-1}B_{L}^{'}\phi_{t}(\tilde{p}^{L})
-B_{F}R_{FF}^{-1}B_{F}^{'}\phi_{t}(\hat{p}^{F})\Big\rangle dt\\
&\leq \mathbb{E}\left|\phi_{0}(X)\right|^{2}
+C_{1}\mathbb{E}\int_{0}^{T}e^{-Kt}\left|\phi_{t}(X)\right|^{2}dt\\
&~~+\mathbb{E}\int_{0}^{T}\left\langle e^{-Kt}\phi_{t}(\tilde{p}^{L}), B_{L}R_{LL}^{-1}B_{L}^{'}\phi_{t}(\tilde{p}^{L})\right\rangle dt\\
&~~+\mathbb{E}\int_{0}^{T}\left\langle e^{-Kt}\phi_{t}(\hat{p}^{F}),B_{F}R_{FF}^{-1}B_{F}^{'}\phi_{t}(\hat{p}^{F})\right\rangle dt,
\end{aligned}
\end{equation}
where $C_{1}=\textrm{tr}(AA^{'})+\textrm{tr}(B_{L}R_{LL}^{-1}B_{L}^{'})+\textrm{tr}(B_{F}R_{FF}^{-1}B_{F}^{'})+1$. So (\ref{fores2}) holds by taking $K=C_{1}+1$.

Suppose \textbf{A1b} holds. From (\ref{dual2}), it always has a proper positive constant $C_{2}$ such that
\begin{equation}\label{fores3}
\begin{aligned}
&\mathbb{E}\int_{0}^{T}e^{-Kt}\left|\phi_{t}(X)\right|^{2}dt\\
&\leq \mathbb{E}\left|\phi_{0}(X)\right|^{2}
+C_{2}\mathbb{E}\left(\left|\phi_{0}(X)\right|\cdot\left|\phi_{0}(p^{L})\right|\right),
\end{aligned}
\end{equation}
thus, $\phi_{t}(X)=0$, $t\in [0,T]$. The uniqueness of $p^{F}$ follows from the
classic estimate for BSDEs below~\cite{1997Backward}:
\begin{equation}\label{back1}
\begin{aligned}
&\mathbb{E}\sup_{t\in[0,T]}
\left|\phi_{t}(p^{F})\right|^{2}+
\mathbb{E}\int_{0}^{T}\left[\left|\phi_{t}(q^{F})\right|^{2}+\left|\phi_{t}(r^{F})\right|^{2}\right]dt\\
&\leq
C_{3}\left(\mathbb{E}\left|G_{F}\phi_{T}(X)\right|+\mathbb{E}\int_{0}^{T}\left|Q_{F}\phi_{t}(X)\right|dt\right),
\end{aligned}
\end{equation}
where $C_{3}$ is a positive constant.  Therefore, we have from the uniqueness of state process $X$ that
$\phi_{t}(p^{F})=0$, $\phi_{t}(q^{F})=0$, $\phi_{t}(r^{F})$, $t\in [0,T]$. Finally, we
prove the uniqueness of the adjoint process $p^{L}$ and the state process $Y$. To achieve it, consider the first and the fifth equations of (\ref{eq3.9}) and compute the respective deviation equations:
\begin{equation}\label{devia2}
\begin{aligned}
\begin{cases}
d\phi_{t}(p^{L})=-\left[A^{'}\phi_{t}(p^{L})+Q_{F}\phi_{t}(Y)\right]dt\\
~~~~~~+\phi_{t}(q^{L})dW_{t}+\phi_{t}(k^{L})dV_{t}^{1}+\phi_{t}(r^{L})dV_{t}^{2},\\
~d\phi_{t}(Y)=\left(A\phi_{t}(Y)-B_{F}R_{FF}^{-1}B_{F}^{'}\phi_{t}(\hat{p}^{L})\right)dt,\\
~\phi_{T}(p^{L})=G_{F}\phi_{T}(Y),~~\phi_{0}(Y)=0,
\end{cases}
\end{aligned}
\end{equation}
where there are no influence to be caused by $\{\phi_{t}(X), t\in[0,T]\}$ and $\{\phi_{t}(p^{F}), t\in[0,T]\}$ since the processes $X$ and $p^{F}$ have been shown to be unique. In another word, we take
$X^{1}=X^{2}$ and $p^{F1}=p^{F2}$.

Applying It\^{o}'s formula to $\langle\phi_{t}(p^{L}), \phi_{t}(Y)\rangle$ and then taking expectation, it yields
\begin{equation}\label{dual3}
\begin{aligned}
&\mathbb{E}\langle G_{F}\phi_{T}(Y), \phi_{T}(Y)\rangle+\mathbb{E}\int_{0}^{T}\Big(
\langle Q_{F}\phi_{t}(Y), \phi_{t}(Y)\rangle\\
&~~+\langle B_{F}R_{FF}^{-1}B_{F}^{'}\phi_{t}(\hat{p}^{L}), \phi_{t}(\hat{p}^{L})\rangle\Big)dt\\
&=\mathbb{E}\langle \phi_{0}(p^{L}),\phi_{0}(Y)\rangle\leq \mathbb{E}(|\phi_{0}(p^{L})|\cdot |\phi_{0}(Y)|).
\end{aligned}
\end{equation}

A similar analysis with the above proofs leads to $\phi_{t}(Y)=0$, and then $\phi_{t}(p^{L})=0$, $\phi_{t}(q^{L})=0$, $\phi_{t}(k^{L})=0$, $\phi_{t}(r^{L})=0$, $t\in[0,T]$.
\end{proof}

\begin{lem}\label{lem2}
Let assumptions \textbf{A1a}, \textbf{A1b} be satisfied. Then, a
necessary and sufficient condition for $U^{L}$ to be an open-loop optimal strategy of the leader's problem described by (\ref{eq3.5})-(\ref{eq3.6}) is
\begin{equation}\label{eq3.7}
\begin{aligned}
U_{t}^{L\star}=-R_{L}^{-1}B_{L}^{'}\tilde{p}^{L}_{t},~~~~t\in [0,T],
\end{aligned}
\end{equation}
with $\tilde{p}_{t}^{L}=\mathbb{E}\left[p_{t}^{L}|\mathcal{Z}_{t}^{L}\right]$. Here,
$\left(p^{L}_{t},\left(q^{L}_{t},k^{L}_{t},r_{t}^{L}\right)\right)$ is computed from the unique solution of the FB-SDE (\ref{eq3.9}).
\end{lem}

\begin{proof}
Based on Lemma~\ref{lem0001}, now we are able to show arguments of this lemma.  The necessary part is left to Appendix.  We show the sufficient part. Let $U^{L\star}$ be the optimal strategy for the leader and  $U^{L}$ be any other. We use $X^{\star}:=X(U^{L\star})$ and $X(U^{L})$ to distinguish the state processes under the two strategies above respectively, the similar denotations will be used for the processes $p^{F}, q^{F}, r^{F}$  and so on.

Applying It\^{o}'s formula to
$\langle\psi_{t}(X),p^{L}_{t}\rangle-\langle Y_{t}, \psi_{t}(p^{F})\rangle$, it gives
\begin{equation}
\begin{aligned}\label{suffi1}
&~~~\langle \psi_{T}(X),G_{L}X_{T}^{\star}\rangle\\
&=\int_{0}^{T}\Big(
\langle B_{L}^{'}\tilde{p}^{L}_{t}, \psi_{t}(U^{L})\rangle
-
\langle Q_{L}X_{t}^{\star},\psi_{t}(X)\rangle\\
&~-
\langle B_{F}R_{FF}^{-1}R_{LF}
R_{FF}^{-1}B_{F}^{'} \hat{p}_{t}^{F\star}, \psi_{t}(\hat{p}^{F})\rangle\Big)dt
\end{aligned}
\end{equation}
where $\psi_{t}(U^{L}):=U^{L}_{t}-U^{L\star}_{t}$, $\psi_{t}(X):=X_{t}(U^{L})-X_{t}(U^{L\star})$ and $\psi_{t}(p^{F}):=p^{F}_{t}(U^{L})-p^{F}_{t}(U^{L\star})$.

The convexity of the integrand  of  $\mathcal{J}^{L}$ with respect to their respective variables leads to
\begin{align}\label{suffi2}
&~~~~\mathcal{J}^{L}(U^{L},U^{F\star})
-\mathcal{J}^{L}(U^{L\star},U^{F\star})\nonumber\\
&\geq \mathbb{E}\langle \psi_{T}(X),G_{L}X_{T}^{\star}\rangle\nonumber\\
&~+\mathbb{E}\int_{0}^{T}\Big( \langle Q_{LL}X_{t}^{\star},\psi_{t}(X)\rangle+\langle R_{LL} U_{t}^{L\star}, \psi_{t}(U^{L})\rangle\\
&~~+\langle B_{F}R_{FF}^{-1}R_{LF}
R_{FF}^{-1}B_{F}^{'} \hat{p}_{t}^{F\star}, \psi_{t}(\hat{p}^{F})\rangle\Big)dt.\nonumber
\end{align}

Inserting the relationship (\ref{suffi1}) into (\ref{suffi2}), we have
\begin{align}\label{suffi3}
&~~~\mathcal{J}^{L}(U^{L},U^{F\star})
-\mathcal{J}^{L}(U^{L\star},U^{F\star})\nonumber\\
&~\geq
\mathbb{E}\int_{0}^{T}\left \langle R_{LL} U_{t}^{L\star}+B_{L}^{'}\tilde{p}^{L}_{t}, \psi_{t}(U^{L})\right\rangle dt.
\end{align}

Hence, sufficiency of (\ref{eq3.7}) is shown.
\end{proof}

\begin{rem}
With respect to the existence of the FB-SDE (\ref{eq3.9}), we show it from the optimal control theory. It could be proven from the stochastic dynamic system point of view. The works in this aspect may refer to  Antonelli~\cite{1993Antonelli}, Hu~{ \em et. al}~\cite{1995Hu}, Yong~\cite{Yong1997},
Peng~{ \em et. al}~\cite{Peng1999}, etc. We point out that their models do not contain ours, it still needs new method to solve it. In
the  procedure of finding the solution,  we are not able to  compute the relationship between the adjoint state and the state directly by the traditional methods, e.g. the four-step method~\cite{1994Ma} or by the nonlinear Feynman-Kac formula~\cite{Yong1999}. From the optimal control point of view,  the reason is that the Riccati equations appearing in the control and filtering parts are coupled. The randomness of one of them will affect another, certainly,  one is
deterministic, so is the other. To verify if it is deterministic or not is what we are interested in.

\end{rem}

\section{Optimal Strategies}

By grouping variables $X_t$ with $Y_t$ and $p_t^L$ with $p_t^F$, we turn the expression (\ref{eq3.9}) to a more compact form
\begin{equation}\label{eq3.12}
\begin{aligned}
\begin{cases}
d\textbf{X}_{t}=\Big(\textbf{A}\textbf{X}_{t}+\tilde{\textbf{S}}\tilde{\textbf{p}}_{t}+\hat{\textbf{S}}
\hat{\textbf{p}}_{t}\Big)dt+\textbf{D}dW_{t},\\
~d\textbf{p}_{t}=-\Big[\textbf{A}^{'}\textbf{p}_{t}+\textbf{Q}\textbf{X}_{t}\Big]dt+\textbf{q}_{t}dW_{t}\\
~~~~~~~~~+\textbf{k}_{t}dV_{t}^{1}+\textbf{r}_{t}dV_{t}^{2},\\
~d\textbf{Z}_{t}=\textbf{H}\textbf{X}_{t}dt+d\textbf{V}_{t},\\
~dZ_{t}^{2}=\textbf{H}_{F}\textbf{X}_{t}dt+dV_{t}^{2},\\
~~\textbf{X}_{0}=\gamma,~~\textbf{p}_{T}=\textbf{G}\textbf{X}_{T},~~\textbf{Z}_{0}=0,~~Z_{0}^{2}=0,
\end{cases}
\end{aligned}
\end{equation}
where
\begin{equation}\label{nota1}
\begin{aligned}
\textbf{X}_{t}&=\mathrm{col}(
     X_{t};
     Y_{t}),~~
\textbf{p}_{t}=\mathrm{col}(
     p_{t}^{L};
     p_{t}^{F}),\\~~
\textbf{q}_{t}&=\mathrm{col}(
     q_{t}^{L};
     q_{t}^{F}),
     ~~\textbf{k}_{t}=\mathrm{col}(
     k_{t}^{L};
     0),\\~~
~~\textbf{r}_{t}&=\mathrm{col}(r_{t}^{L};r_{t}^{F}), ~~\gamma=\mathrm{col}(
                      \xi;
                     0),\\
\textbf{A}&=\mathrm{diag}(A, A), ~~\tilde{\textbf{S}}=\mathrm{diag}(-B_{L}R_{LL}^{-1}B^{'}_{L},0),\\
~~\textbf{D}&=\mathrm{col}(D;0),~~~~\textbf{H}_{L}=(H_{1},0),\\
                  \hat{\textbf{S}}&=\left(
                    \begin{array}{cc}
                      0 & -B_{F}R_{FF}^{-1}B^{'}_{F} \\
                      -B_{F}R_{FF}^{-1}B^{'}_{F} & B_{F}R_{FF}^{-1}R_{LF}R_{FF}^{-1}B_{F}^{'} \\
                    \end{array}
                  \right),\\~~\textbf{Q}&=\left(
                    \begin{array}{cc}
                      Q_{L} & Q_{F} \\
                      Q_{F} & 0 \\
                    \end{array}
                  \right),~~\textbf{G}=\left(
                    \begin{array}{cc}
                      G_{L} & G_{F} \\
                      G_{F} & 0 \\
                    \end{array}
                  \right),\\
\textbf{H}&=\left(
                    \begin{array}{cc}
                      H_{1} & 0 \\
                      H_{2} & 0 \\
                    \end{array}
                  \right),~~\textbf{H}_{F}=(H_{2},0),\\
                  ~~\textbf{Z}_{t}&=\mathrm{col}(
     Z_{t}^{1};
     Z_{t}^{2}),~~\textbf{V}_{t}=\mathrm{col}(
     V_{t}^{1};
     V_{t}^{2}).
\end{aligned}
\end{equation}
Moreover, $\tilde{\textbf{p}}_{t}:=\mathbb{E}\left[\textbf{p}_{t}|\mathcal{Z}_{t}^{L}\right]$ and $\hat{\textbf{p}}_{t}:=\mathbb{E}\left[\textbf{p}_{t}|\mathcal{Z}_{t}^{F}\right]$.
\bigskip

Inspired by the ordinary differential equation theory, the main idea is to change a two-point boundary value problem to a Cauchy initial value problem~\cite{Zill2009}. To do it, some functional relationships between the adjoint state variables and the state variables need to be found, relating $\textbf{p}$ with $\textbf{X}$, $\tilde{\textbf{p}}$ with $\tilde{\textbf{X}}$, $\hat{\textbf{p}}$ with $\hat{\textbf{X}}$, three layers, where $\tilde{\textbf{X}}:=\{\tilde{\textbf{X}}_{t}, t\in[0,T]\}$ and $\hat{\textbf{X}}:=\{\hat{\textbf{X}}_{t}, t\in[0,T]\}$ with $\tilde{\textbf{X}}_{t}:=\mathbb{E}[\textbf{X}_{t}|\mathcal{Z}_{t}^{L}]$ and $\hat{\textbf{X}}_{t}:=\mathbb{E}[\textbf{X}_{t}|\mathcal{Z}_{t}^{F}]$. Among them, the last two parts are enough to design feedback type strategy, we state it   in the following Theorem.

\begin{thm}\label{lem3b}
Let assumptions \textbf{A1a}, \textbf{A1b} be satisfied. With notations in (\ref{nota1}), the feedback type strategies of the leader and the follower given by (\ref{eq3.3}) and (\ref{eq3.7}) are given, respectively, by
\begin{equation}\label{c17}
\begin{aligned}
U_{t}^{F\star}=&-R^{-1}_{FF}B_{F}^{'}\Big[P_{t,21},P_{t,22}\Big]\hat{\textbf{X}}_{t},
\end{aligned}
\end{equation}
and
\begin{equation}\label{0c17}
\begin{aligned}
U_{t}^{L\star}=&-R^{-1}_{LL}B_{L}^{'}
\Big[P_{t,11}^{\dag},P_{t,12}^{\dag}\Big]\tilde{\textbf{X}}_{t}\\
&-R^{-1}_{LL}B_{L}^{'}\Big[P_{t,11}-P_{t,11}^{\dag},
P_{t,12}-P_{t,12}^{\dag}\Big]\hat{\textbf{X}}_{t},
\end{aligned}
\end{equation}
where $P_{t,ij},~P^{\dag}_{t,ij}, ~i,j=1,2$ are, respectively, the $ij$-element of the $2\times2$-blocked matrices $P_{t}$ and $P^{\dag}_{t}$, computed by (\ref{c19}) below
\begin{equation}\label{c19}
\begin{aligned}
&\dot{P}_{t}^{\dag}+\textbf{A}^{'}P_{t}^{\dag}+P_{t}^{\dag}\textbf{A}+P_{t}^{\dag}
\tilde{\textbf{S}}P_{t}^{\dag}+\textbf{Q}\\
&~~~~+P^{\ddag}_{t}\Xi_{t,11}\textbf{H}_{F}^{'}\textbf{H}_{F}=0,
~~P_{t}=P_{t}^{\dag}+P_{t}^{\ddag},\\
&\dot{P}_{t}+\textbf{A}^{'}P_{t}+P_{t}\textbf{A}+P_{t}(\tilde{\textbf{S}}
+\hat{\textbf{S}})P_{t}+\textbf{Q}=0,\\
&\dot{\Sigma}_{t}=\textbf{A}\Sigma_{t}+\Sigma_{t}\textbf{A}^{'}
-\Sigma_{t}\textbf{H}^{'}\textbf{H}\Sigma_{t}+
\textbf{D}\textbf{D}^{'},\\
&\dot{\Xi}_{t}=F_{t}\Xi_{t}+\Xi_{t}F_{t}^{'}+\check{D}\check{D}^{'}
+G^{\dag}_{t}G^{\dag'}_{t}
+G^{\ddag}_{t}G^{\ddag'}_{t}\\
&~~~~-(\Xi_{t}\check{H}^{'}_{F}+G^{\ddag}_{t})(\check{H}_{F}\Xi_{t}
+G^{\ddag'}_{t}),\\
&P_{T}^{\dag}=\textbf{G},~~P_{T}=\textbf{G},
~~\Sigma_{0}=\mathrm{cov}(\gamma,\gamma),\\
&~\Xi_{0}=\left[\begin{array}{cc}
\Sigma_{0} & \Sigma_{0} \\
 \Sigma_{0} & \Sigma_{0} \\
 \end{array}
 \right],
\end{aligned}
\end{equation}
with $\Xi_{t,11}$ being the $11$-element of the $2\times2$-blocked matrix $\Xi_{t}$ and
\begin{equation}\label{c22}
\begin{aligned}
F_{t}&=\left[
               \begin{array}{cc}
                \textbf{ A}+\tilde{\textbf{S}}P^{\dag}_{t} & -\tilde{\textbf{S}}P_{t}^{\dag} \\
                 0 & \textbf{A}-\Sigma_{t}\textbf{H}^{'}\textbf{H} \\
               \end{array}
             \right], \\
             \check{D}&=\mathrm{col}(\textbf{D};\textbf{D}),
             ~~G^{\dag}_{t}=\mathrm{col}(0;-\Sigma_{t}\textbf{H}^{'}_{L}),\\ ~~G^{\ddag}_{t}&=\mathrm{col}(0;-\Sigma_{t}\textbf{H}^{'}_{F}), ~~\check{H}_{F}=\left(\left[H_{2},0\right],\left[0,0\right]\right).
\end{aligned}
\end{equation}
In addition, the augmented system state estimates $\hat{\textbf{X}}_{t}$ and  $\tilde{\textbf{X}}_{t}$ appeared in (\ref{c17}) and (\ref{0c17}) are given, respectively, by the two Kalman-Bucy filters
\begin{equation}\label{c21}
\begin{aligned}
\begin{cases}
d\hat{\textbf{X}}_{t}=\Big([\textbf{A}+(\tilde{\textbf{S}}+\hat{\textbf{S}})P_{t}]
\hat{\textbf{X}}_{t}\Big)dt+\Xi_{t,11}\textbf{H}^{'}_{F}d\textbf{I}_{t}^{F},\\
~d\textbf{I}_{t}^{F}=dZ_{t}^{2}-\textbf{H}_{F}\hat{\textbf{X}}_{t}dt,
~~\hat{\textbf{X}}_{0}=\mathbb{E}[\gamma],
\end{cases}
\end{aligned}
\end{equation}
and
\begin{equation}\label{c20}
\begin{aligned}
\begin{cases}
d\tilde{\textbf{X}}_{t}=\Big([\textbf{A}
+\tilde{\textbf{S}}P^{\dag}_{t}]\tilde{\textbf{X}}_{t}+[(\tilde{\textbf{S}}+\hat{\textbf{S}})P_{t}
-\tilde{\textbf{S}}P^{\dag}_{t}]\hat{\textbf{X}}_{t}\Big)dt\\
~~~~~~~~~+\Sigma_{t}\textbf{H}^{'}d\textbf{I}_{t},\\
~d\textbf{I}_{t}=d\textbf{Z}_{t}-\textbf{H}\tilde{\textbf{X}}_{t}dt,
~~\tilde{\textbf{X}}_{0}=\mathbb{E}[\gamma].
\end{cases}
\end{aligned}
\end{equation}
\end{thm}

\begin{proof}
(\textbf{Inner Layer Calculation})
We seek the relationship between $\hat{\textbf{p}}_{t}$ and $\hat{\textbf{X}}_{t}$ in this step.  We first calculate $\textbf{p}_{t}$ from the second equation of (\ref{eq3.12}). By a simple calculation, it reads
\begin{equation}\label{in1}
\begin{aligned}
\textbf{p}_{t}&=\Phi^{'}(T,t)\textbf{G}\textbf{X}_{T}
+\int_{t}^{T}\Phi^{'}(s,t)\textbf{Q}\textbf{X}_{s}ds\\
&-\int_{t}^{T}\Phi^{'}(s,t)\Big[\textbf{q}_{s}dW_{s}
+\textbf{k}_{s}dV_{s}^{1}+\textbf{r}_{s}dV_{s}^{2}\Big],
\end{aligned}
\end{equation}
where the state transition  matrix $\Phi^{'}(s,t)$ satisfies
\begin{equation}\label{cc22}
\begin{aligned}
\begin{cases}
\frac{\partial \Phi^{'}(s,t)}{\partial t}=-\textbf{A}^{'}\Phi^{'}(s,t),\\
\Phi^{'}(s,s)=I_{id}.
\end{cases}
\end{aligned}
\end{equation}

Next, we calculate the estimate $\hat{\textbf{p}}_{t}$. For simplicity we use
\begin{align*}
\hat{\textbf{X}}_{s|t}:=\mathbb{E}[\textbf{X}_{s}|\mathcal{Z}_{t}^{F}], ~~~~ \hat{\textbf{p}}_{s|t}:=\mathbb{E}[\textbf{p}_{s}|\mathcal{Z}_{t}^{F}],~~~~s\geq t,
\end{align*}
to denote the predictions of $\textbf{X}_{s}$ and $\textbf{p}_{s}$ with respect to $\mathcal{Z}_{t}^{F}$.

Taking conditional expectation, with respect to $\mathcal{Z}_{t}^{F}$, on both sides of (\ref{in1}), it yields
\begin{equation}\label{in2}
\begin{aligned}
\hat{\textbf{p}}_{t}&=\Phi^{'}(T,t)\textbf{G}\hat{\textbf{X}}_{T|t}
+\int_{t}^{T}\Phi^{'}(s,t)
\textbf{Q}\hat{\textbf{X}}_{s|t}ds.
\end{aligned}
\end{equation}

Noticing the left-hand side of (\ref{in2}) involved the prediction terms $\hat{\textbf{X}}_{T|t}$ and $\hat{\textbf{X}}_{s|t}$, we are therefore going to calculate the estimate $\hat{\textbf{X}}_{t}$. To do it, consider the signal model constructed by the first and the forth equations of (\ref{eq3.12}). Based on Theorem 6.6
\cite{Jazwinski1970}, this estimate can be obtained by the following nonlinear filter equation
\begin{equation}\label{b16}
\begin{aligned}
\begin{cases}
d\hat{\textbf{X}}_{t}=\Big(\textbf{A}\hat{\textbf{X}}_{t}
+[\tilde{\textbf{S}}+\hat{\textbf{S}}]\hat{\textbf{p}}_{t}\Big)dt+\Upsilon_{t}\textbf{H}^{'}_{F}d\textbf{I}_{t}^{F},\\
d\textbf{I}_{t}^{F}:=dZ_{t}^{2}-\textbf{H}_{F}\hat{\textbf{X}}_{t}dt,\\
~\Upsilon_{t}:=\mathbb{E}[\textbf{X}_{t}\textbf{X}_{t}^{'}|\mathcal{Z}_{t}^{F}]
-\hat{\textbf{X}}_{t}\hat{\textbf{X}}_{t}^{'}.
\end{cases}
\end{aligned}
\end{equation}

Based on the filter (\ref{b16}), we now are able to derive the prediction terms stated as above. Utilizing $\mathbb{E}\Big[\mathbb{E}[\textbf{X}_{s}|\mathcal{Z}_{s}^{F}]\Big|\mathcal{Z}_{t}^{F}\Big]
=\mathbb{E}[\textbf{X}_{s}|\mathcal{Z}_{t}^{F}]$,
~$s\geq t$, the predictor, referring to Section~3.7~\cite{Ahmed1998Linear}, is given by
\begin{equation}\label{b17}
\begin{aligned}
d\hat{\textbf{X}}_{s|t}&=\Big(\textbf{A}\hat{\textbf{X}}_{s|t}
+[\tilde{\textbf{S}}+\hat{\textbf{S}}]\hat{\textbf{p}}_{s|t}\Big)ds,\\
~~~~~~\hat{\textbf{X}}_{t|t}&=\hat{\textbf{X}}_{t},~~s\geq t.
\end{aligned}
\end{equation}

Observing (\ref{in2}) and (\ref{b17}), it hints us to let
\begin{equation}\label{b18a}
\begin{aligned}
\varphi_{t}=\hat{\textbf{p}}_{t}-P_{t}\hat{\textbf{X}}_{t},~~~~t\in[0,T],
\end{aligned}
\end{equation}
where $P, \varphi$ are two  processes to be determined.

Finally, we use a undetermined coefficient method to pin down $P, \varphi$ above. Applying (\ref{b18a}) to (\ref{b17}), the explicit  solution of the predictor is
\begin{equation}\label{b21}
\begin{aligned}
\hat{\textbf{X}}_{s|t}&=\int_{t}^{s}
\Psi(s,\tau)\Big[(\tilde{\textbf{S}}+\hat{\textbf{S}})\varphi_{\tau}\Big]d\tau+\Psi(s,t)\hat{\textbf{X}}_{t},
\end{aligned}
\end{equation}
where the state transition matrix $\Psi(s,t)$ is given by
\begin{equation}\label{b20}
\begin{aligned}
\begin{cases}
\frac{\partial \Psi(s,t)}{\partial s}=(\textbf{A}+[\tilde{\textbf{S}}
+\hat{\textbf{S}}]P_{s})\Psi(s,t),\\
~\Psi(t,t)=I_{id}.
\end{cases}
\end{aligned}
\end{equation}

Inserting (\ref{b21}) into (\ref{in2}) and comparing with (\ref{b18a}), it can be verified that
\begin{align*}
P_{t}&=\int_{t}^{T}\Phi^{'}(s,t)\textbf{Q}\Psi(s,t)ds+\Phi^{'}(T,t)\textbf{G}
\Psi(T,t),\\
\varphi_{t}&=\Phi^{'}(T,t)\textbf{G}\int_{t}^{T}
\Psi(T,s)(\tilde{\textbf{S}}+\hat{\textbf{S}})\varphi_{s}ds\\
&+\int_{t}^{T}\int_{t}^{s}\Phi^{'}(s,t)\textbf{Q}\Psi(s,\tau)\Big[(\tilde{\textbf{S}}+\hat{\textbf{S}})
\varphi_{\tau}\Big]d\tau ds.
\end{align*}

Taking partial derivation with respect to the time variable $t$, it yields
the second equation of (\ref{c19}) and $\varphi_{t}\equiv0$, ~$t\in [0,T]$.

(\textbf{Outer Layer Calculation})~
From now on, the first three equations of (\ref{eq3.12}) have changed into
\begin{equation}\label{b25}
\begin{aligned}
\begin{cases}
d\textbf{X}_{t}=\Big(\textbf{A}\textbf{X}_{t}+\tilde{\textbf{S}}\tilde{\textbf{p}}_{t}+\hat{\textbf{S}}
P_{t}\hat{\textbf{X}}_{t}\Big)dt+\textbf{D}dW_{t},\\
d\textbf{p}_{t}=-\Big[\textbf{A}^{'}\textbf{p}_{t}+\textbf{Q}\textbf{X}_{t}\Big]dt+\textbf{q}_{t}dW_{t}\\
~~~~~~~~+\textbf{k}_{t}dV_{t}^{1}+\textbf{r}_{t}dV_{t}^{2},\\
d\textbf{Z}_{t}=\textbf{H}\textbf{X}_{t}dt+d\textbf{V}_{t},\\
~~\textbf{X}_{0}=\gamma,~~p_{T}=\textbf{G}\textbf{X}_{T},~~\textbf{Z}_{0}=0.
\end{cases}
\end{aligned}
\end{equation}

We look for the relationship between $\tilde{\textbf{p}}_{t}$ and $\tilde{\textbf{X}}_{t}$ in this step, its proof is similar to the inner layer part. Via (\ref{b25}), we compute these two variables. For simplicity, denote by
\begin{align*}
\tilde{\textbf{X}}_{s|t}:=\mathbb{E}[\textbf{X}_{s}|\mathcal{Z}_{t}^{L}],~~~~
\tilde{\textbf{p}}_{s|t}:=\mathbb{E}[\textbf{p}_{s}|\mathcal{Z}_{t}^{L}],~~~~
s\geq t,
\end{align*}
the predictions of $\textbf{X}_{s}$ and $\textbf{p}_{s}$ with respect to $\mathcal{Z}_{t}^{L}$.

From (\ref{in1}), a similar procedure with (\ref{in2}), the conditional mean $\tilde{\textbf{p}}_{t}$ is given by
\begin{equation}\label{in3}
\begin{aligned}
\tilde{\textbf{p}}_{t}&=\Phi^{'}(T,t)\textbf{G}\tilde{\textbf{X}}_{T|t}+\int_{t}^{T}\Phi^{'}(s,t)
\textbf{Q}\tilde{\textbf{X}}_{s|t}ds.
\end{aligned}
\end{equation}

Now, we are in a position to calculate $\tilde{\textbf{X}}_{t}$. Consider the signal model constructed by the first and the third
equations of (\ref{b25}).  By Kalman-Bucy filtering theory~\cite{Jazwinski1970}, it is given by
\begin{equation}\label{in4}
\begin{aligned}
\begin{cases}
d\tilde{\textbf{X}}_{t}=\Big(\textbf{A}\tilde{\textbf{X}}_{t}+\tilde{\textbf{S}}\tilde{\textbf{p}}_{t}+\hat{\textbf{S}}
P_{t}\hat{\textbf{X}}_{t}\Big)dt+\Sigma_{t}\textbf{H}^{'}d\textbf{I}_{t},\\
~d\textbf{I}_{t}=d\textbf{Z}_{t}-\textbf{H}\tilde{\textbf{X}}_{t}dt,
\end{cases}
\end{aligned}
\end{equation}
where the initial data and $\Sigma$ are the same with the ones in (\ref{c20}) and (\ref{c19}) respectively.

Observing (\ref{in3}) and (\ref{in4}), it hints us to let
\begin{equation}\label{in5}
\begin{aligned}
\tilde{\textbf{p}}_{t}=P^{\dag}_{t}\tilde{\textbf{X}}_{t}
+P^{\ddag}_{t}\hat{\textbf{X}}_{t},
\end{aligned}
\end{equation}
where $P^{\dag}$ is a deterministic matrix-valued  process to be determined and note that
$P^{\ddag}=P-P^{\dag}$.

Next, we write the first system equation in (\ref{b25}) and the
  filer equations (\ref{b16}) and (\ref{in4}) together. It is in a form of
\begin{small}
\begin{equation}\label{in1112}
\begin{aligned}
\left(
  \begin{array}{c}
  d\textbf{X}_{t}\\
    d\tilde{\textbf{X}}_{t} \\
    d\hat{\textbf{X}}_{t} \\
  \end{array}
\right)
&=\Theta_{t}\left(
  \begin{array}{c}
  \textbf{X}_{t}\\
    \tilde{\textbf{X}}_{t} \\
    \hat{\textbf{X}}_{t} \\
  \end{array}
\right)dt
+\left(
            \begin{array}{c}
            \textbf{D}\\
              0 \\
              0 \\
            \end{array}
          \right)dW_{t}\\
          &~~~~+\left(
            \begin{array}{c}
            0\\
              \Sigma_{t}\textbf{H}^{'} \\
              0 \\
            \end{array}
          \right)d\textbf{I}_{t}+\left(
                              \begin{array}{c}
                              0\\
                                0 \\
                                \Upsilon_{t}\textbf{H}^{'}_{F} \\
                              \end{array}
                            \right)dV_{t}^{2},
\end{aligned}
\end{equation}
with
\begin{align}\label{rep1111}
\Theta_{t}=\left[
    \begin{array}{ccc}
     \textbf{A} & \tilde{\textbf{S}}P_{t}^{\dag} & \tilde{\textbf{S}}P_{t}^{\ddag}+\hat{\textbf{S}}P_{t}\\
      0 & \textbf{A}+\tilde{\textbf{S}}P_{t}^{\dag} & \tilde{\textbf{S}}P_{t}^{\ddag}+\hat{\textbf{S}}P_{t} \\
      \Upsilon_{t}\textbf{H}^{'}_{F}\textbf{H}_{F}& 0 & \textbf{A}
+[\tilde{\textbf{S}}+\hat{\textbf{S}}]P_{t}-\Upsilon_{t}\textbf{H}^{'}_{F}\textbf{H}_{F} \\
    \end{array}
  \right],
\end{align}
\end{small}
where the observation equation (the forth equation in (\ref{eq3.12}) ) is used to substitute the innovation term $\textbf{I}_{t}^{F}$ with the observation noise $V^{2}_{t}$.

Define  the following state transition matrix
\begin{align}\label{st11}
\begin{cases}
\frac{\partial \Pi(s,t)}{\partial s}=\Theta_{s}\Pi(s,t),~~~~s\geq t,\\
  \Pi(t,t)=I_{id}.
  \end{cases}
\end{align}

Via (\ref{st11}), the prediction of the state variable of (\ref{in1112}) is expressed as:
\begin{align}\label{in2222}
\left(
  \begin{array}{c}
  \mathbb{E}[\textbf{X}_{s}|\mathcal{Z}_{t}^{L}]\\
    \mathbb{E}[\tilde{\textbf{X}}_{s}|\mathcal{Z}_{t}^{L}] \\
  \mathbb{E}[\hat{\textbf{X}}_{s}|\mathcal{Z}_{t}^{L}] \\
  \end{array}
\right)=\Pi(s,t)\left(
  \begin{array}{c}
  \tilde{\textbf{X}}_{t}\\
    \tilde{\textbf{X}}_{t} \\
  \hat{\textbf{X}}_{t} \\
  \end{array}
\right),~~s\geq t,
\end{align}
where the last three stochastic integrals disappear, since the independent increment property of the $\{\mathcal{Z}_{t}^{L}, t\in[0,T]\}$-adapted
Wiener processes $\{\textbf{I}_{t}, t\in[0,T]\}$ and the following equalities:
\begin{align*}
&\mathbb{E}\left[\int_{t}^{s}\Pi(s,\tau)[\cdot]dW_{\tau}\Big|\mathcal{Z}_{t}^{L}\right]\\
&=\mathbb{E}\left[\mathbb{E}\left[\int_{t}^{s}\Pi(s,\tau)
[\cdot]dW_{\tau}\Big|\mathcal{F}_{t}\right]\bigg|\mathcal{Z}_{t}^{L}\right]=0,\\
&\mathbb{E}\left[\int_{t}^{s}\Pi(s,\tau)[\cdot]dV^{2}_{\tau}|\mathcal{Z}_{t}^{L}\right]\\
&=\mathbb{E}\left[\mathbb{E}\left[\int_{t}^{s}\Pi(s,\tau)
[\cdot]dV^{2}_{\tau}\Big|\mathcal{F}_{t}\right]\bigg|\mathcal{Z}_{t}^{L}\right]=0.
\end{align*}
In the above, the tower property of the conditional expectation has been used.

Applying (\ref{in2222}) to (\ref{in3}) and comparing with (\ref{in5}), we have
\begin{align}\label{para2222}
\left(P_{t}^{1,\dag},P_{t}^{2,\dag}, P_{t}^{\ddag}\right)&=\Phi^{'}(T,t)\textbf{G}\left[0,I_{id}, 0\right]\Pi(T,t)\nonumber\\
&+\int_{t}^{T}\Phi^{'}(s,t)\textbf{Q}\left[0,I_{id}, 0\right]\Pi(s,t)ds,
\end{align}
where $P_{t}^{\dag}=P_{t}^{1,\dag}+P_{t}^{2,\dag}$, $t\in [0,T]$.

Taking partial derivation with respect to the time variable $t$, it yields:
\begin{align}\label{para3333}
&\dot{P}_{t}^{1,\dag}+\textbf{A}^{'}P_{t}^{1,\dag}+P_{t}^{1,\dag}\textbf{A}+
P_{t}^{\ddag}\Upsilon_{t}\textbf{H}_{F}^{'}\textbf{H}_{F}=0,\\
&\dot{P}_{t}^{2,\dag}+\textbf{A}^{'}P_{t}^{2,\dag}+P_{t}^{2,\dag}\textbf{A}+
P_{t}^{\dag}\tilde{\textbf{S}}P_{t}^{\dag}+\textbf{Q}=0,\\
&P_{T}^{1,\dag}=0,~~P_{T}^{2,\dag}=\textbf{G}.\nonumber
\end{align}

Thus, $P_{t}^{\dag}$ satisfies
\begin{align}\label{pa44444}
&\dot{P}_{t}^{\dag}+\textbf{A}^{'}P_{t}^{\dag}+P_{t}^{\dag}\textbf{A}
+P_{t}^{\dag}\tilde{\textbf{S}}P_{t}^{\dag}+
P_{t}^{\ddag}\Upsilon_{t}\textbf{H}_{F}^{'}\textbf{H}_{F}+\textbf{Q}=0,\\
&P_{T}^{\dag}=\textbf{G}.\nonumber
\end{align}

Finally, we calculate the filtering equations for $\tilde{\textbf{X}}$ and $\hat{\textbf{X}}$. Applying (\ref{in5}) to (\ref{in4}),  it yields (\ref{c20}). We then compute the filtering
$\hat{\textbf{X}}$, it is done by the Kalman-Bucy filtering theory. Define an error variable $e_{t}:=\textbf{X}_{t}-\tilde{\textbf{X}}_{t}$ in prior.
Considering the first and the forth equations of (\ref{eq3.12}) and (\ref{in4}), it can be
verified that $\textbf{X}$, $e$ and $Z^{2}$ construct the following signal system:
\begin{equation}\label{c23}
\begin{aligned}
\begin{cases}
d\textbf{X}_{t}=\Big(\textbf{A}\textbf{X}_{t}+\tilde{\textbf{S}}P^{\dag}_{t}\tilde{\textbf{X}}_{t}+[
(\tilde{\textbf{S}}+\hat{\textbf{S}})P_{t}-\tilde{\textbf{S}}P^{\dag}_{t}]\\
~~~~~~~~~~~~~~~~\times\hat{\textbf{X}}_{t}\Big)dt+\textbf{D}dW_{t},\\
~~de_{t}=(\textbf{A}-\Sigma_{t}\textbf{H}^{'}\textbf{H})e_{t}+\textbf{D}dW_{t}
-\Sigma_{t}\textbf{H}^{'}_{L}dV_{t}^{1}\\
~~~~~~~~~~~~~~~~-\Sigma_{t}\textbf{H}^{'}_{F}dV_{t}^{2},\\
~dZ_{t}^{2}=\textbf{H}_{F}\textbf{X}_{t}dt+dV_{t}^{2},\\
~~~\textbf{X}_{0}=\gamma,~~e_{0}=\gamma-\mathbb{E}[\gamma],~~Z_{0}^{2}=0.
\end{cases}
\end{aligned}
\end{equation}

Via (\ref{c22}), they can be rewritten, in a compact form, as
\begin{equation}\label{c24}
\begin{aligned}
\begin{cases}
~dx_{t}=\left(F_{t}x_{t}+\mathcal{L}(\hat{\textbf{X}}_{t})\right)dt+\check{D}dW_{t}+G_{t}^{\dag}dV_{t}^{1}\\
~~~~~~~~~~~~~~~~+G_{t}^{\ddag}dV_{t}^{2},\\
\mathcal{L}(\hat{\textbf{X}}_{t})=\mathrm{col}\left(\left[
(\tilde{\textbf{S}}+\hat{\textbf{S}})P_{t}-\tilde{\textbf{S}}P^{\dag}_{t}\right]\hat{\textbf{X}}_{t};0\right),\\
dZ^{2}_{t}=\check{H}_{F}x_{t}dt+dV_{t}^{2},\\
~~x_{0}=\mathrm{col}(\gamma;\gamma-\mathbb{E}[\gamma]),~~~~Z_{0}^{2}=0,
\end{cases}
\end{aligned}
\end{equation}
where $x_{t}:=\mathrm{col}(\textbf{X}_{t};e_{t})$.

It is a signal model with common noise. Based on the Kalman-Bucy filtering theory~\cite{Xiong2008}, its estimator equation is given by
\begin{equation}\label{c25}
\begin{aligned}
\begin{cases}
d\hat{x}_{t}=\left(F_{t}\hat{x}_{t}+\mathcal{L}(\hat{\textbf{X}}_{t})\right)dt
+\left(G_{t}^{\ddag}+\Xi_{t}\check{H}_{F}^{'}\right)\\
~~~~~~~~~~~~~~~\times\left(dZ^{2}_{t}-\check{H}_{F}\hat{x}_{t}dt\right),\\
~~\hat{x}_{0}=\mathrm{col}\left(\mathbb{E}[\gamma];0\right),
\end{cases}
\end{aligned}
\end{equation}
where $\hat{x}_{t}:=\mathbb{E}[x_{t}|\mathcal{Z}^{F}_{t}]$ and the covariance matrix $\Xi$ is introduced in (\ref{c19}).

Note also that $\hat{e}_{t}=\hat{\textbf{X}}_{t}-\mathbb{E}[\tilde{\textbf{X}}_{t}|\mathcal{Z}^{F}_{t}]=0$, $t\in[0,T]$.
In other word, the second coordinate of the state estimate of (\ref{c25}) equals zero. To simplify (\ref{c25}), it yields (\ref{c21}). Comparing (\ref{c21}) with (\ref{b16}), we get $\Upsilon_{t}=\Xi_{t,11}$, $t\in [0,T]$.
Then, (\ref{pa44444}) is exactly the first equation of (\ref{c19}).
\end{proof}

\section{Application Examples}

In this section we give some application examples of our main result.

\textbf{Example}~1.~(Government Debt Stabilization Problem~\cite{Aarle1995})~We consider the following differential game, on government debt stabilization, with the fiscal authority acting as Stackelberg leader
\begin{equation}\label{c26}
\begin{aligned}
dd_{t}&=(rd_{t}+f_{t}-m_{t})dt,\\
dz_{t}^{1}&=\rho_{1} d_{t}dt+dV^{1}_{t},\\
dz_{t}^{2}&=\rho_{2}d_{t}dt+dV^{2}_{t},
\end{aligned}
\end{equation}
where the government debt $d_{t}$ is the state variable, and the issue of base money $m_{t}$ and primary fiscal deficits $f_{t}$, are adjusted by the monetary authority and the fiscal authority respectively. $\{z^{1}_{t}\}$ and $\{z^{2}_{t}\}$ are two observation processes, with $\{z^{2}_{t}\}$ being available by the monetary authority and $\{z^{1}_{t},z^{2}_{t}\}$ being available by the fiscal authority.  We assume that these two authorities rely on their respective observation information to implement policies.

The objective of the fiscal authority is to minimize a sum of the primary fiscal deficit, base money growth and government debt
\begin{equation}\label{c028}
\begin{aligned}
\mathcal{J}^{L}(f,m)&=\mathbb{E}\int_{0}^{T}\frac{1}{2}\Big(\lambda(d_{t}-\bar{d})^{2}
+(f_{t}-\bar{f})^{2}\\
&~~~~+\eta (m_{t}-\bar{m})^{2}\Big)dt,
\end{aligned}
\end{equation}
where $\bar{f}$, $\bar{m}$ and $\bar{d}$ represent exogenous policy targets for base money growth, the primary fiscal deficit and public debt, respectively.

The monetary authority sets the growth of base money so as to minimize the following loss function
\begin{equation}\label{c27}
\begin{aligned}
\mathcal{J}^{F}(f,m)=\mathbb{E}\int_{0}^{T}\frac{1}{2}\Big(\kappa(d_{t}-\bar{d})^{2}
+(m_{t}-\bar{m})^{2}\Big)dt,
\end{aligned}
\end{equation}
where $\frac{1}{\kappa}(>0)$ measures how conservative the central bank is with respect to the money growth.

Letting $x^{1}_{t}:=d_{t}-\bar{d}$, $x^{2}_{t}:=r\bar{d}+\bar{f}-\bar{m}$, $X_{t}:={\rm col}(x^{1}_{t};x^{2}_{t})$, $U^{F}_{t}=m_{t}-\bar{m}$, $U^{L}_{t}=f_{t}-\bar{f}$, $Z_{t}^{1}:=z_{t}^{1}-(\rho_{1}\bar{d})t$, $Z_{t}^{2}:=z_{t}^{2}-(\rho_{2}\bar{d})t$, (\ref{c26}), (\ref{c028}) and (\ref{c27}) can be written in the form~(\ref{eq2.1a}) and (\ref{eq2.1b}), with
\begin{align*}
A&=\textrm{diag}(r,0),~B_{F}=\textrm{col}(-1;0),~B_{L}=\textrm{col}(1;0),\\
D&=0,~H_{1}=[\rho_{1},0],~H_{2}=[\rho_{2},0],\\
Q_{F}&=\textrm{diag}(\kappa,0),~R_{FL}=0,~R_{FF}=1,~G_{F}=0,\\
Q_{L}&=\textrm{diag}(\lambda,0),~R_{LL}=1,~R_{LF}=\eta,~G_{L}=0.
\end{align*}
Clearly, they satisfy assumptions \textbf{A1a}, \textbf{A1b}. We can then calculate the optimal strategies for the monetary authority and the fiscal authority using Theorem~\ref{lem3b}.

\textbf{Example}~2. ~(Linear Optimal Servo Problem~\cite{Rreindler1969}) Consider a signal model formed by the state and measurement equations of two plants (or agents)
\begin{equation}\label{c28}
\begin{aligned}
\textrm{Plant}~1:~dx^{1}_{t}&=(A_{1}x^{1}_{t}+B_{1}U^{L}_{t})dt+D_{1}dW_{t}^{1},\\
\textrm{Plant}~2:~dx^{2}_{t}&=(A_{1}x^{2}_{t}+B_{2}U^{F}_{t})dt+D_{2}dW_{t}^{2},\\
dZ_{t}^{1}&=h_{1}\begin{bmatrix}
                   x^{1}_{t}  \\
                   x^{2}_{t} \\
                 \end{bmatrix}
 dt+dV^{1}_{t},\\
dZ_{t}^{2}&=h_{2}\begin{bmatrix}
                   x^{1}_{t}  \\
                   x^{2}_{t} \\
                 \end{bmatrix}dt+dV^{2}_{t}.
\end{aligned}
\end{equation}
We assume that Plant~1 has access to more information, and therefore acts as the Stackelberg leader. More precisely, the linear noisy measurement $\{Z_{t}^{2}\}$ is available to Plant~2, whereas $\{Z_{t}^{1},Z_{t}^{2}\}$ are available to Plant~1.

The objective of Plant~2 is to track a linear combination of the state of Plant~1 and a command input
\begin{equation}\label{c29}
\begin{aligned}
\mathcal{J}^{F}(U^{L},U^{F})&=\mathbb{E}\int_{0}^{T}
\frac{1}{2}\Big[|x^{2}_{t}-(\Gamma_{21}x^{1}_{t}+\Gamma_{22}s_{t})|^{2}\\
&~~~~+|U^{F}_{t}|^{2}\Big]dt,
\end{aligned}
\end{equation}
where the command input  $s_{t}$, generated by a command generator, is described by~\cite{Rreindler1969}
\begin{equation}\label{c30}
\begin{aligned}
ds_{t}=Ls_{t}dt.
\end{aligned}
\end{equation}
The above setup includes as a particular case the one in which the command input is a prescribed time-function~\cite{Rreindler1969}, which is often called a tracking problem.

Plant~2 has a similar objective function, given by
\begin{equation}\label{c31}
\begin{aligned}
\mathcal{J}^{L}(U^{L},U^{F})&=\mathbb{E}\int_{0}^{T}
\frac{1}{2}\Big[|x^{1}_{t}-(\Gamma_{12}x^{2}_{t}+\Gamma_{11}s_{t})|^{2}\\
&~~~~+\theta|U^{F}_{t}|^{2}+(1-\theta)|U^{L}_{t}|^{2}\Big]dt,
\end{aligned}
\end{equation}
where $\theta\in(0,1)$.

Defining $X_{t}:={\rm col}(x^{1}_{t};x^{2}_{t};s_{t})$, $W_{t}:={\rm col}(W^{1}_{t};W^{2}_{t})$, then (\ref{c28}), (\ref{c29}), (\ref{c30}), (\ref{c31}) can be written in the form~(\ref{eq2.1a}) and (\ref{eq2.1b}), where
\begin{align*}
A&=\textrm{diag}(A,A,L),~B_{F}=\textrm{col}(0;B_{2};0),\\
B_{L}&=\textrm{col}(B_{1};0;0),~H_{1}=[h_{1},0],~H_{2}=[h_{2},0],\\
R_{FL}&=0,~R_{FF}=I,~G_{F}=0,\\
R_{LF}&=\theta I,~R_{LL}=(1-\theta)I,~G_{L}=0,
\end{align*}
and
\begin{small}
\begin{align*}
D&=\begin{bmatrix}
     D_{1} & 0 \\
     0 & D_{2} \\
     0 & 0 \\
   \end{bmatrix},~Q_{F}=\begin{bmatrix}
                          \Gamma_{21}^{'}\Gamma_{21} & -\Gamma^{'}_{21} & \Gamma_{21}^{'}\Gamma_{22} \\
                          -\Gamma_{21} & I & -\Gamma_{22} \\
                          \Gamma_{22}^{'}\Gamma_{21} & -\Gamma_{22}^{'} & \Gamma_{22}^{'}\Gamma_{22}\\
                        \end{bmatrix},\\
    Q_{L}&=\begin{bmatrix}
                          I & -\Gamma_{12} & -\Gamma_{11} \\
                          -\Gamma^{'}_{12} & \Gamma^{'}_{12}\Gamma_{12} & \Gamma^{'}_{12}\Gamma_{11} \\
                          -\Gamma_{11}^{'} & \Gamma_{11}^{'}\Gamma_{12} & \Gamma_{11}^{'}\Gamma_{11}\\
                        \end{bmatrix}.
\end{align*}
\end{small}
It is easy to check that assumptions \textbf{A1a}, \textbf{A1b} are satisfied. Theorem~\ref{lem3b} then can be used to derive the optimal control strategies.

\section{Conclusion and Future work}\label{conclusion}

In this work we studied the LQG differential Stackelberg game under nested observation information pattern. We provided conditions to guarantee existence  of the solution and recasted the original problem as that of solving a new FB-SDE. Using this result we gave explicit forms for the leader's and follower's control strategies.
The possible extension includes the related problem of the infinite horizon case or  the follower containing multiple plants.  Further research contains dealing with more general information patterns, referring to~\cite{Charalambos2017Noisy},\cite{Chamralambous2016}, etc. and nonlinear dynamic system model with finite-dimensional filters, referring to~\cite{V1981Exact,Chamralambous1997Certain,2000tang}, etc.

\bigskip

\section{Acknowledgments}
The authors would like to thank Prof. Juanjuan Xu of Shandong University for her help and relevant discussions.

\section{Appendix}
\textit{Proof of \textbf{Lemma}~\ref{lem1}.}\par
\textit{Necessity}. We split the necessity proof in steps:\\
Step~1. (Convex perturbation) Given $U^L\in \mathcal{U}^L$, for any $\varepsilon\in[0,1]$, we perturb
$U^{F\star}$ to $U^{F,\varepsilon}$ in the following way
\begin{align}\label{aa001}
U^{F,\varepsilon}:=U^{F\star}+\varepsilon \delta U^{F},~~\delta U^{F}:= U^{F}-U^{F\star},
\end{align}
where $U^{F\star}\in \mathcal{U}^F$ and $U^{F} \in \mathcal{U}^F$ are denoted to be the optimal strategy and any other one respectively.

From the convexity of the admissible set $\mathcal{U}^{F}$, it yields that $U^{F,\varepsilon} \in \mathcal{U}^{F}$.

Step~2. (Variation calculation) Rewrite (\ref{eq3.2}) in a compact form
\begin{small}
\begin{align*}
d\Big[
   \begin{array}{c}
     X_{t} \\
     Z_{t}^{2} \\
   \end{array}
 \Big]&=\Big(\Big[
           \begin{array}{cc}
             A & 0 \\
             H_{2} & 0 \\
           \end{array}
         \Big]\Big[
   \begin{array}{c}
     X_{t} \\
     Z_{t}^{2} \\
   \end{array}
 \Big]+\Big[
           \begin{array}{c}
             B_{F} \\
             0 \\
           \end{array}
         \Big]U_{t}^{F}+\Big[
                            \begin{array}{c}
                              B_{L}U^{L}_{t} \\
                              0 \\
                            \end{array}
                          \Big]\Big)dt\\
                          &~~~+\Big[
                                \begin{array}{cc}
                                  D & 0 \\
                                  0 & I \\
                                \end{array}
                              \Big]\Big[
                                       \begin{array}{c}
                                         dW_{t} \\
                                         dV_{t}^{2} \\
                                       \end{array}
                                     \Big].
\end{align*}
\end{small}
Suppose that ${\rm col}(X^{\star},Z^{2\star}):=(X(U^{F\star}), Z^{2}(U^{F\star}))$ and ${\rm col}(X^{\varepsilon},Z^{2,\varepsilon}):=(X(U^{F,\varepsilon}), Z^{2}(U^{F,\varepsilon}))$  denote the augmented state processes under the two strategies above respectively. From the above equations, we can clearly decompose $X^{\varepsilon}$ as $X^{\star}+\varepsilon \delta X$ and $Z^{2,\varepsilon}$ as $Z^{2,\star}+\varepsilon \delta Z^{2}$, where
\begin{small}
\begin{equation}\label{aa2}
\begin{aligned}
d\Big[
   \begin{array}{c}
     \delta X_{t} \\
     \delta Z_{t}^{2} \\
   \end{array}
 \Big]&=\Big(\Big[
           \begin{array}{cc}
             A & 0 \\
             H_{2} & 0 \\
           \end{array}
         \Big]\Big[
   \begin{array}{c}
     \delta X_{t} \\
     \delta Z_{t}^{2} \\
   \end{array}
 \Big]+\Big[
           \begin{array}{c}
             B_{F} \\
             0 \\
           \end{array}
         \Big]\delta U_{t}^{F}\Big)dt.
\end{aligned}
\end{equation}
\end{small}

In another hand, calculate the G\^{a}teaux differential of $\mathcal{J}^{F}$ at $U^{F\star}$ in the direction $\delta U^{F}$
\begin{equation}\label{aa3}
\begin{aligned}
&~~~~d\mathcal{J}^{F}(\cdot,U^{F\star};\delta U^{F})\\
&=\lim_{\varepsilon \rightarrow 0}
\frac{\mathcal{J}^{F}(\cdot,U^{F\star}+\varepsilon \delta U^{F})
-\mathcal{J}^{F}(\cdot,U^{F\star})}{\varepsilon}\\
&=\mathbb{E}\int_{0}^{T}\Big(\langle Q_{F}X_{t}^{\star}, \delta X_{t}\rangle
+\langle R_{FF}U_{t}^{F\star}, \delta U_{t}^{F}\rangle\Big) dt\\
&~~+\mathbb{E}\langle G_{F} X_{T}^{\star}, \delta X_{T}\rangle
\end{aligned}
\end{equation}

Step~3. (Terminal condition conversion) Applying It\^{o}'s formula to $t\mapsto \langle p^{F}_{t}, \delta X_{t}\rangle+
\langle \underline{p}^{F}_{t},\delta Z_{t}^{2}\rangle$, with $p^{F}$ and $\underline{p}^{F}$ being two processes to be determined, it yields
\begin{equation}\label{aa4}
\begin{aligned}
\int_{0}^{T}d\Big(\langle p^{F}_{t}, \delta X_{t}\rangle+
\langle \underline{p}^{F}_{t},\delta Z_{t}^{2}\rangle\Big)
=\langle p^{F}_{T}, \delta X_{T}\rangle+
\langle \underline{p}^{F}_{T},\delta Z_{T}^{2}\rangle
\end{aligned}
\end{equation}

The main idea is to substitute the terminal term of (\ref{aa3}) with the right hand of (\ref{aa4}). Thus, we pin down the terminal value $p^{F}_{T}=G_{F} X_{T}^{\star}$
and $\underline{p}^{F}_{T}=0$. Applying (\ref{aa2}) to the left hand side of (\ref{aa4}), then inserting (\ref{aa4}) into (\ref{aa3}), it can be verified that
\begin{equation}\label{aa5}
\begin{aligned}
&~~~~d\mathcal{J}^{F}(\cdot,U^{F\star};\delta U^{F})\\
&=\mathbb{E}\int_{0}^{T}\Big(\langle dp^{F}_{t}+A^{'}p^{F}_{t}dt+
Q_{F}X_{t}^{\star}dt, \delta X_{t}\rangle\\
&~~+\langle d\underline{p}^{F}_{t}+H^{2'}\underline{p}^{F}_{t}dt,
\delta Z^{2}_{t}\rangle\\
&~~~+\langle B_{F}^{'}p^{F}_{t}+R_{FF}U_{t}^{F\star}, \delta U_{t}^{F}\rangle dt\Big).
\end{aligned}
\end{equation}

From (\ref{aa5}), it hint us, by the arbitrariness of $\delta X$ and $\delta Z^{2}$,  that
\begin{equation}\label{aa6}
\begin{aligned}
&dp^{F}_{t}+A^{'}p^{F}_{t}dt+Q_{F}X_{t}^{\star}dt=0,~~p^{F}_{T}=G_{F} X_{T}^{\star},\\
&d\underline{p}^{F}_{t}+H^{2'}\underline{p}^{F}_{t}dt=0,~~\underline{p}^{F}_{T}=0.
\end{aligned}
\end{equation}

But the state processes of the above two equations does not satisfy the adaptiveness. The BSDEs theory provided a valid method to deal with it~\cite{Yong1999}. For detail,  modify them to the following BSDEs
\begin{equation}\label{aa7}
\begin{aligned}
&dp^{F}_{t}=-[A^{'}p^{F}_{t}+Q_{F}X_{t}^{\star}]dt+q^{F}_{t}dW_{t}
+r^{F}_{t}dV_{t}^{2},\\
&d\underline{p}^{F}_{t}=-H^{2'}\underline{p}^{F}_{t}dt
+\underline{q}^{F}_{t}dW_{t}+\underline{r}^{F}_{t}dV_{t}^{2},\\
&p^{F}_{T}=G_{F} X_{T}^{\star},~~\underline{p}^{F}_{T}=0.
\end{aligned}
\end{equation}
In the above, the martingale representation theorem is used to correct the adaptiveness, Chapter~7~\cite{Yong1999}. This is why four stochastic integrals appear.

Notice that $\underline{p}^{F}=0$, $\underline{q}^{F}=0$, $\underline{r}^{F}=0$ satisfy the second equation. From the uniqueness of the solution of the linear BSDE, referring to the same chapter as above, it  is exactly the unique solution. Thus, (\ref{eq3.4}) is obtained.

Step~4. ($\sigma$-sub-algebra projection) Now, (\ref{aa5}) has changed to
\begin{equation}\label{aa8}
\begin{aligned}
&~~~~d\mathcal{J}^{F}(\cdot,U^{F\star};\delta U^{F})\\
&=\mathbb{E}\int_{0}^{T}\langle B_{F}^{'}p^{F}_{t}+R_{FF}U_{t}^{F}, \delta U_{t}^{F}\rangle dt\\
&=\mathbb{E}\int_{0}^{T}\langle \mathbb{E}[B_{F}^{'}p^{F}_{t}+R_{FF}U_{t}^{F\star}|\mathcal{Z}_{t}^{F}], \delta U_{t}^{F}\rangle dt,
\end{aligned}
\end{equation}
where the tower property of conditional expectation is used in the second equality because the follower
implements strategy relying on the $\sigma$-sub-algebra $\mathcal{Z}_{t}^{F}$ at
time $t$. Finally, the arbitrariness of $\delta U^{F}$ leads to (\ref{eq3.3}).

\textit{Sufficiency} Given $U^{L}\in \mathcal{U}^{L}$, letting $U^{F\star}$ and $U^{F}$ be the same as before, $X^{\star}$ and $X$ denote the corresponding state respectively.  We need calculate the difference $\mathcal{J}^{F}(\cdot,U^{F})-\mathcal{J}^{F}(\cdot,U^{F\star})$. The convexity of $\mathcal{J}^{F}$ with respect to the state and strategy variables lead to
\begin{equation}\label{aa9}
\begin{aligned}
&~~~\mathcal{J}^{F}(\cdot,U^{F})-\mathcal{J}^{F}(\cdot,U^{F\star})\\
&~\geq \mathbb{E}\int_{0}^{T}\Big(\langle Q_{F}X^{\star}_{t},\psi_{t}(X)\rangle+\langle R_{FF}U^{F\star}_{t},\psi_{t}(U^{F})\rangle\Big)dt\\
&~~+\mathbb{E}\langle G_{F}X_{T}^{\star},\psi_{T}(X)\rangle,
\end{aligned}
\end{equation}
where $\psi_{t}(X):=X_{t}-X_{t}^{\star}$, $\psi_{t}(U^{F}):= U^{F}_{t}-U^{F\star}_{t}$.

Next, applying It\^{o}'s formula to $t\mapsto \langle p^{F}_{t}, \psi_{t}(X)\rangle$, it gives
 \begin{equation}\label{aa10}
\begin{aligned}
&~~~\langle p^{F}_{T},\psi_{T}(X)\rangle=\langle p^{F}_{0},\psi_{0}(X)\rangle\\
&~+
\int_{0}^{T}\langle -Q_{F}X_{t}^{\star},\psi_{t}(X)\rangle+\langle p^{F}_{t}, B_{F}\psi_{t}(U^{F})\rangle dt.
\end{aligned}
\end{equation}

Inserting (\ref{aa10}) into (\ref{aa9}) and  noticing $\psi_{0}(X)=0$, it yields
$\mathcal{J}^{F}(\cdot,U^{F})\geq \mathcal{J}^{F}(\cdot,U^{F\star})$.
Sufficiency of (\ref{eq3.3}) is verified.

\bigskip

\textit{Proof of \textbf{Lemma}~\ref{lem2}.}\par
\textit{Necessity}. Its proof is similar to the part of necessity proof of Lemma~\ref{lem1}. We only give a framework. Firstly, a convex perturbation, similar with (\ref{aa001}), is taken for $U^{L}$. The G\^{a}teaux differential of $\mathcal{J}^{L}$ at $U^{L\star}$ in the direction $\delta U^{L}$ is
\begin{equation}\label{aa11}
\begin{aligned}
&~~~~d\mathcal{J}^{L}(U^{L\star},U^{F\star};\delta U^{L})\\
&=\lim_{\varepsilon \rightarrow 0}
\frac{\mathcal{J}^{L}(U^{L\star}+\varepsilon \delta U^{L\star},U^{F\star})
-\mathcal{J}^{L}(U^{L\star},U^{F\star})}{\varepsilon}\\
&=\mathbb{E}\int_{0}^{T}\Big(\langle Q_{L}X_{t}^{\star}, \delta X_{t}\rangle+\langle R_{LL}U_{t}^{L\star}, \delta U_{t}^{L}\rangle\\
&~~+\langle R_{LF}R_{FF}^{-1}B_{F}^{'}\hat{p}_{t}^{F\star}, R_{FF}^{-1}B_{F}^{'}\delta \hat{p}_{t}^{F}\rangle\Big) dt\\
&~~~+\mathbb{E}\langle G_{L} X_{T}^{\star}, \delta X_{T}\rangle,
\end{aligned}
\end{equation}
where $X^{\star}:=X(U^{L\star})$ and $p^{F\star}:=p^{F}(U^{L\star})$ are the state processes corresponding with $U^{L\star}$.

Since there are two equations of $X$'s and $p^{F}$'s in (\ref{eq3.6}), we need to
apply It\^{o}'s formula to $t\mapsto \langle\delta X_{t},p_{t}^{L}\rangle-
\langle \delta p^{F}_{t}, Y_{t}\rangle$ to substitute the terminal term of
(\ref{aa11}), where $p^{L}$ and $Y$ are two processes to be determined. After doing it, may refer to (\ref{suffi1}), it can be checked that
\begin{equation}\label{aa12}
\begin{aligned}
&~~~~d\mathcal{J}^{L}(U^{L\star},U^{F\star};\delta U^{L})\\
&=\mathbb{E}\int_{0}^{T}\langle B_{L}^{'}p_{t}^{L}+R_{LL}U_{t}^{L\star}, \delta U_{t}^{L}\rangle dt.
\end{aligned}
\end{equation}

Finally, a similar reason with (\ref{aa8}) leads to the desired result.

\bibliographystyle{plain}

\bibliography{sgame}

\end{document}